\documentclass[oneside]{amsart}
\usepackage[utf8]{inputenc}
\usepackage{amsmath,amstext,amsopn,amssymb, amsfonts, amsthm, mathtools}
\usepackage[dvipsnames]{xcolor}
\usepackage[margin=1in]{geometry}
\usepackage{float}
\usepackage{bbm}
\usepackage{mathscinet}
\usepackage{wrapfig}
\usepackage{pdfsync}
\usepackage{tikz}
\usetikzlibrary{calc}
\usetikzlibrary {arrows.meta}
\usepackage{enumitem}
\usepackage{caption}

\usepackage{subcaption}
\usepackage{xfrac}

\usepackage{graphicx}
\usepackage{mathrsfs}
\usepackage[ruled,vlined,algosection]{algorithm2e}
\DeclareMathAlphabet{\mathpzc}{OT1}{pzc}{m}{it}

\usepackage[pdftex,bookmarks,colorlinks,breaklinks]{hyperref}  
\definecolor{dullmagenta}{rgb}{0.4,0,0.4}   
\definecolor{darkblue}{rgb}{0,0,0.4}
\definecolor{darkgreen}{rgb}{0,0.4,0}
\definecolor{Azure}{rgb}{0.0, 0.5, 1.0}
\hypersetup{linkcolor=darkblue,citecolor=blue,filecolor=dullmagenta,urlcolor=darkblue} 

\definecolor{green0}{rgb}{0.7,1,0.7}

\definecolor{green1}{rgb}{0,0.8,0}
\definecolor{red1}{rgb}{0.8,0,0}

\definecolor{green2}{rgb}{0,0.5,0}
\definecolor{red2}{rgb}{0.5,0,0}

\newtheorem{TheoremLetter}{Theorem}
{}

\numberwithin{equation}{section}

\newtheorem{theorem}{Theorem}
\newtheorem*{theorem*}{Theorem}
\numberwithin{theorem}{section}

\newtheorem{definition}{Definition}
\newtheorem*{definition*}{Definition}
\numberwithin{definition}{section}

\newtheorem*{conjecture*}{Conjecture}
\numberwithin{conjecture}{section}

\newtheorem*{question*}{Question}
\numberwithin{question}{section}

\newtheorem*{example*}{Example}
\numberwithin{example}{section}

\newtheorem{lemma}[theorem]{Lemma}
\newtheorem*{lemma*}{Lemma}

\newtheorem{proposition}[theorem]{Proposition}
\newtheorem*{proposition*}{Proposition}

\newtheorem{corollary}[theorem]{Corollary}
\newtheorem*{corollary*}{Corollary}

\newtheorem*{remark*}{Remark}
\numberwithin{remark}{section}

\makeatletter
\newcommand{\customlabel}[2]{%
   \protected@write \@auxout {}{\string \newlabel {#1}{{#2}{\thepage}{#2}{#1}{}} }%
   \hypertarget{#1}{#2}
}
\makeatother

\def\XXint#1#2#3{{\setbox0=\hbox{$#1{#2#3}{\int}$}
     \vcenter{\hbox{$#2#3$}}\kern-.5\wd0}}

\newcommand{\car}{\operatorname{Car}}

\newcommand{\ind}{\mathbbm{1}}

\newcommand{\candidateB}{\mathbf{B}}
\newcommand{\candidatef}{\mathbf{f}}
\newcommand{\candidateg}{\mathbf{g}}

\newcommand{\mbr}{\mathbb{R}}

\newcommand{\mcb}{\mathcal{B}}

\newcommand{\bff}{\mathbf{f}}
\newcommand{\bfg}{\mathbf{g}}
\newcommand{\bfb}{\mathbf{B}}

\newcommand{\iLint}{\mathtt{L}}

\definecolor{darkmagenta}{rgb}{0.55, 0.0, 0.55}
\definecolor{darkspringgreen}{rgb}{0.09, 0.45, 0.27}
\definecolor{indiagreen}{rgb}{0.07, 0.53, 0.03}
\definecolor{indigo}{rgb}{0.29, 0.0, 0.51}
\definecolor{orange-red}{rgb}{1.0, 0.27, 0.0}

\definecolor{irinared}{rgb}{0.671,0.247,0.153}
\definecolor{irinared2}{rgb}{0.89,0.239,0.239}
\definecolor{irinablue}{rgb}{0.153,0.392,0.671}

\newcommand{\polygon}{\scalebox{1.5}{$\diamond$}}

\makeindex
\begin{document}

\author{Irina Holmes Fay}
\address[Irina Holmes Fay]{Texas A\&M University}
\email{irinaholmes@tamu.edu}
\thanks{I.H.F. is supported by NSF grant NSF-DMS-2246985. This grant also supported the second author's visit to College Station in November 2023, 
during which part of this work was completed.}

\author{Guillermo Rey}
\address[Guillermo Rey]{Universidad Aut\'onoma de Madrid}
\email{guillermo.rey@uam.es}
\thanks{G.R. is supported by grant PID2022-139521NA-I00, funded by MICIU/AEI/10.13039/501100011033}

\author{Kristina Ana Škreb}
\address[Kristina Ana Škreb]{University of Zagreb}
\email{kristina.ana.skreb@grad.unizg.hr}
\thanks{K.A.\v S  was supported in part by the Croatian Science Foundation under the project number HRZZ-IP-2022-10-5116 (FANAP)}

\title{Sharp restricted weak-type estimates for sparse operators}
\begin{abstract}
  We find the exact Bellman function associated to the level-sets of sparse operators acting on characteristic functions.
\end{abstract}
\maketitle

\section{Introduction}
\label{S:1}

In this article we study the restricted weak-type bound for positive sparse operators, i.e.: operators of the form
\begin{align*}
  \mathcal{A}_{\mathcal{S}} f = \sum_{J \in \mathcal{S}} \langle |f| \rangle_J \ind_J,
\end{align*}
where $\langle f \rangle_J$ is the average of $f$ over the interval $J$: $\langle f \rangle_J = \frac{1}{|J|}\int_J f$,
and $\mathcal{S}$ is a \emph{sparse} collection of dyadic intervals. We use $|\cdot|$ to denote the Lebesgue measure on $\mathbb{R}$.

The notion of \emph{sparse} collection is a generalization of the concept of a pairwise disjoint collection.
A sparse collection is in a certain sense ``almost disjoint''.
In particular: we say that a collection $\mathcal{S}$ of measurable sets is \emph{sparse}
if for every $J \in \mathcal{S}$ there exists a subset $E(J) \subseteq J$ such that the collection
$\{E(J)\}_{J \in \mathcal{S}}$ is pairwise-disjoint,
and such that $|E(J)| \geq \frac{1}{2}|J|$ for all $J \in \mathcal{S}$.

Sparse operators such as $\mathcal{A}_{\mathcal{S}}$ model the more delicate Calderón-Zygmund operators,
and in fact one can dominate a Calderón-Zygmund operator by an appropriately chosen positive sparse operator.
We refer the reader to \cite{LNBook} and \cite{MR4311199} for more details on this topic.

In the present article we will be concerned with \emph{local} level-set estimates, the most well-known of which is
the weak-type $(1,1)$ inequality
\begin{align} \label{intro:weak-type}
  |\{x \in [0,1):\, \mathcal{A}_{\mathcal{S}}f(x) \geq \lambda\}| \lesssim \frac{1}{\lambda} \int_0^1 f,
\end{align}
whenever the sparse collection $\mathcal{S}$ consists entirely of dyadic subintervals of $[0,1)$ and $f \geq 0$.

One can prove this inequality using a Calderón-Zygmund decomposition in the usual way,
see also \cite{MR3161110} for another approach closer in spirit to this article.

We will study the particular case of \eqref{intro:weak-type} where the function $f$ is assumed to be the characteristic function of a subset of $[0,1)$,
that is: the \emph{restricted} weak-type setting. Under this stronger hypothesis one should expect additional decay in $\lambda$.

A particularly instructive example appears when $E = [0,1)$.
In this case the function $x \mapsto \mathcal{A}_{\mathcal{S}}\ind_E(x)$ counts the number of intervals in $\mathcal{S}$ that contain the point $x$,
and we have
\begin{align*}
  |\{x \in [0,1):\, \mathcal{A}_{\mathcal{S}}\ind_{[0,1)}(x) \geq \lambda\}| \leq 2^{2-\lambda}.
\end{align*}
One can show exponential decay estimates like this by duality --for example, using sharp bounds for the operator norm of the dyadic maximal function on $L^p$ for small $p$--
but obtaining the best constants requires a different approach.

Instead of trying to indirectly obtain sharp upper bounds, we take the Bellman function approach (see for example \cite{MR0744226}, \cite{MR1685781}, \cite{MR2354322}, where this technique is used).
This strategy shifts the focus to the seemingly harder question of computing the function given by
\begin{align} \label{intro:first_sup}
  \mathbb{B}(x, \lambda) = \sup |\{t \in [0,1):\, \mathcal{A}_{\mathcal{S}}\ind_E(t) \geq \lambda\}|,
\end{align}
where the supremum is taken over all dyadic sparse collections $\mathcal{S}$, and all subsets $E \subseteq [0,1)$ with measure $x$.

In this article we find an explicit expression for $\mathbb{B}$.
Our analysis allows us to show that the supremum is attained, and to describe the pairs $(E, \mathcal{S})$ at which it is attained.

The explicit expression of $\mathbb{B}$ is slightly involved, but one can easily derive the following estimate from Theorem \ref{MainTheorem} below.
\begin{corollary*}
  If $\mathcal{S}$ is sparse, $E \subseteq [0,1)$, $|E| = 2^{-n}$, and $\lambda \in \mathbb{N} - 2^{-n}$ is at least $2$, then
  \begin{align*}
    |\{t \in [0,1):\, \mathcal{A}_{\mathcal{S}}\ind_E(t) \geq \lambda\}| \leq |E| 2^{3-|E|-\lambda},
  \end{align*}
  and equality is attained for certain pairs $(E, \mathcal{S})$.
\end{corollary*}

In order to find the supremum in \eqref{intro:first_sup} -- as is common when employing the Bellman function technique -- we need an extra variable
measuring the degree of ``sparseness'' of $\mathcal{S}$.
To this end, we can introduce the \emph{Carleson height} and \emph{Carleson constant} of a sequence $\{\alpha_{I}\}$ defined on the dyadic subintervals $I$ of $[0,1)$:
\begin{align*}
  A(\alpha, I) := \frac{1}{|I|}\sum_{J \in \mathcal{D}(I)} \alpha_J |J| \qquad \text{and} \qquad \|\alpha\|_{\car} = \sup_{I \in \mathcal{D}([0,1))} A(\alpha, I),
\end{align*}
respectively.
Here we are denoting the collection of all dyadic intervals contained in $I$ by $\mathcal{D}(I)$.
We will also abbreviate $A(\alpha, [0,1))$ to just $A(\alpha)$.

One can identify collections $\mathcal{S} \subseteq \mathcal{D}([0,1))$ as a sequences indexed by $\mathcal{D}([0,1))$ and taking values in $\{0,1\}$. Indeed, setting
\begin{align*}
  \alpha_I = \begin{cases}
    1 &\text{if } I \in \mathcal{S} \\
    0 &\text{otherwise}
  \end{cases}
\end{align*}
it is easy to see that, if $\alpha$ is the sequence associated with a sparse collection, then the Carleson constant of $\alpha$ is bounded by $2$.
Using this identification between binary sequences over $\mathcal{D}([0,1))$ and subcollections of $\mathcal{D}([0,1))$, we will say that a collection $\mathcal{S}$
is a Carleson collection if the associated binary sequence is a Carleson sequence with constant at most $2$.

As we have just mentioned, every sparse collection is easily seen to be Carleson.
The converse is also true, though the proof is not as easy, the reader may find the equivalence in Lemma 6.3 of \cite{LNBook} (one may find alternative proofs in \cite{Verbitsky}, \cite{BarronThesis}, \cite{Hanninen}, or \cite{MR4682731}).
This equivalence will be central to our approach, as the \emph{Carleson height} lends itself more easily to dyadic decompositions.

With these definitions we can introduce: for $(x, A, \lambda) \in [0,1] \times [0,2] \times \mathbb{R}$
\begin{align} \label{intro:full_B}
  \mathbb{B}(x,A,\lambda) = \sup |\{t \in [0,1):\, \mathcal{A}_{\mathcal{S}}\ind_E(t) \geq \lambda\}|,
\end{align}
where the supremum is taken over all measurable subsets $E \subseteq [0,1)$ with measure $x$, and all Carleson collections $\mathcal{S}$ with \emph{Carleson height} $A$.

We should note that it is not immediately obvious that there exist Carleson collections $\alpha$ with $A(\alpha) = A$ for every $A \in [0,2]$.
However, it is not difficult to construct a Carleson collection with height $2$, and one can then combine this with the empty sequence to obtain all the heights in between $0$ and $2$.

In order to compute $\mathbb{B}$ we explicitly find a lower bound $\candidateB$ using the defining dynamics of the problem.
This lower bound is then shown to be optimal by proving that $\mathbf{B}$ satsifies a certain \emph{strengthened concavity} property.

For $\lambda \leq 0$ the function $\candidateB$ is identically $1$, this is just the observation that $\mathcal{A}_{\mathcal{S}}\ind_E$ is never negative.
When $\lambda > 0$, the behavior of the function $\candidateB$ varies slightly depending on whether $\lambda \leq 1$ or $\lambda > 1$.

When $0 < \lambda \leq 1$ we can split the $(x,A)$-domain $R = [0,1] \times [0,2]$ into four regions (see Figure \ref{fig:xaplane}).
These regions are defined in terms of $\lambda$:
\begin{align*}
  \Sigma_0'(\lambda) &= \polygon\big((\tfrac{\lambda}2, 2), (\lambda, 1), (1,1), (1,2)\big) \\
  \Sigma_0(\lambda) &= \polygon\big((0,0), (\lambda, 1), (\tfrac{\lambda}2, 2)\big) \\
   \Sigma_1(\lambda) &= \polygon\big((0,0), (1,0), (1,1), (\lambda, 1)\big) \\
  \Sigma_2(\lambda) &= \polygon\big((0,0), (\tfrac{\lambda}2, 2), (0,2)\big),
\end{align*}
where we have used the notation $\polygon(A,B,C, \dots)$ to denote the polygon given by the points $A, B, C, \dots$ in counterclockwise order.

\begin{figure}[H]
  \begin{tikzpicture}[scale=0.7]
    \draw[thick, -] (0,0) -- (16,0);
    \draw[thick, -] (16,0) -- (16,8);
    \draw[thick, -] (16,8) -- (0,8);
    \draw[thick, -] (0,8) -- (0,0);

    \draw[thick, -] (0,0) -- (10,4);
    \draw[thick, -] (0,0) -- (5,8);
    \draw[thick, -] (10,4) -- (5,8);
    \draw[thick, -] (10,4) -- (16,4);

    \node at  (-0.75, 0) {$(0,0)$};
    \node at  (16+0.75, 0) {$(1,0)$};
    \node at  (16+0.75, 8) {$(1,2)$};
    \node at  (-0.75, 8) {$(0,2)$};

    \fill (10,4) circle (2pt); 
    \fill (5,8) circle (2pt);  

    \node at (10+0.5, 4+0.5) {$(\lambda,1)$};
    \node at (5, 8+0.5) {$(\tfrac{\lambda}2, 2)$};

    \node at (13, 6) {\huge $\Sigma_0'(\lambda)$};
    \node at (13, 2) {\huge $\Sigma_1(\lambda)$};
    \node at (5, 4) {\huge $\Sigma_0(\lambda)$};
    \node at (2, 6) {\huge $\Sigma_2(\lambda)$};
  \end{tikzpicture}
\caption{$(x, A)$-domain regions for $0<\lambda\leq 1$}
\label{fig:xaplane}
\end{figure}
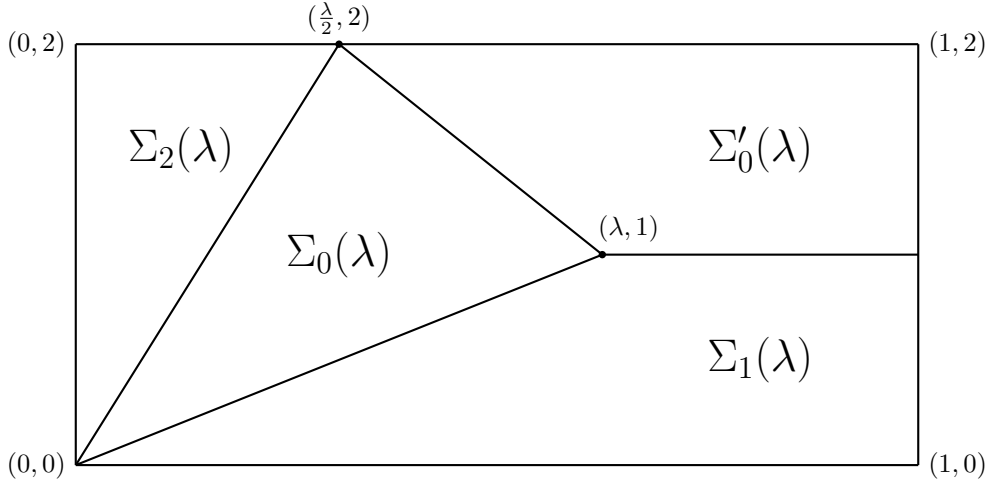

When $0 < \lambda \leq 1$ we define $\candidateB$ as follows:
\begin{align*}
  \candidateB(x,A,\lambda) = \begin{cases}
    1 &\text{in } \Sigma_0' \\
    A &\text{in } \Sigma_1 \\
    \frac{A+\frac{2x}{\lambda}}{3} &\text{in } \Sigma_0 \\
    \frac{A}{2}\candidatef\Big(\frac{2x}{A}, \lambda\Big) &\text{in } \Sigma_2
  \end{cases}
\end{align*}

When $\lambda > 1$ we set
\begin{align*}
  \candidateB(x,A,\lambda) = \begin{cases}
    \frac{A}{2}\candidatef\Big(\frac{2x}{A}, \lambda\Big) &\text{when } 2x \leq A \\
    \frac{A}{2}\candidatef(1, \lambda) &\text{when } 2x > A
  \end{cases}
\end{align*}

Naturally, the function $\candidatef$ is the main character in our article.

We are thus left with specifying $\candidatef$. To this end, for any $0 \leq k \leq m$ define
\begin{align*}
  F_k^m = \bigg(\frac{1}{2^k}, \: m-k+3-\frac{1}{2^{k}}\bigg).
\end{align*}
For every $m \geq 0$ we denote by $\Gamma_m$ the piecewise-linear curve given by the sequence of points
\begin{align*}
  (0,0) \to F_m^m \to F_{m-1}^m \to F_{m-2}^m \to \dots \to F_0^m.
\end{align*}

These curves form the ``scaffolding'' on which $\candidatef$ is defined.
In particular, $\candidatef \equiv 2^{-m}$ on $\Gamma_m$.
Between the curves $\{\Gamma_m\}$ the function $\candidatef$ is defined as the linear interpolation along horizontal lines from these curves.


In particular, through the point $(x_\ast, \lambda_\ast)$ draw a horizontal line $\ell$ and let
$x_0 < x_\ast < x_1$ be the $x$-intercepts of $\ell$ with the two curves $\Gamma_{m_0}$ and $\Gamma_{m_1}$ (respectively) between which $(x_\ast, \lambda_\ast)$ lies.
There will always be an intersecting curve ``above'' $(x_\ast, \lambda_\ast)$, but there may not be one ``below''.
In case there is no curve below, we just set $x_1 = 1$ and $m_1 = m_0$.
Now write $x_\ast$ as a convex combination of $x_0$ and $x_1$: $x_\ast = (1-t)x_0 + t x_1$, i.e.:
\begin{align*}
    t = \frac{x_\ast - x_0}{x_1-x_0}.
\end{align*}
Then
\begin{align*}
    \candidatef(x_\ast, \lambda_\ast) &= (1-t)\candidatef(x_0, \lambda_\ast) + t\candidatef(x_1,\lambda_\ast)\\
    &= (1-t)2^{-m_0} + t2^{-m_1}
\end{align*}

\begin{figure}[H]
  \includegraphics[scale=0.33]{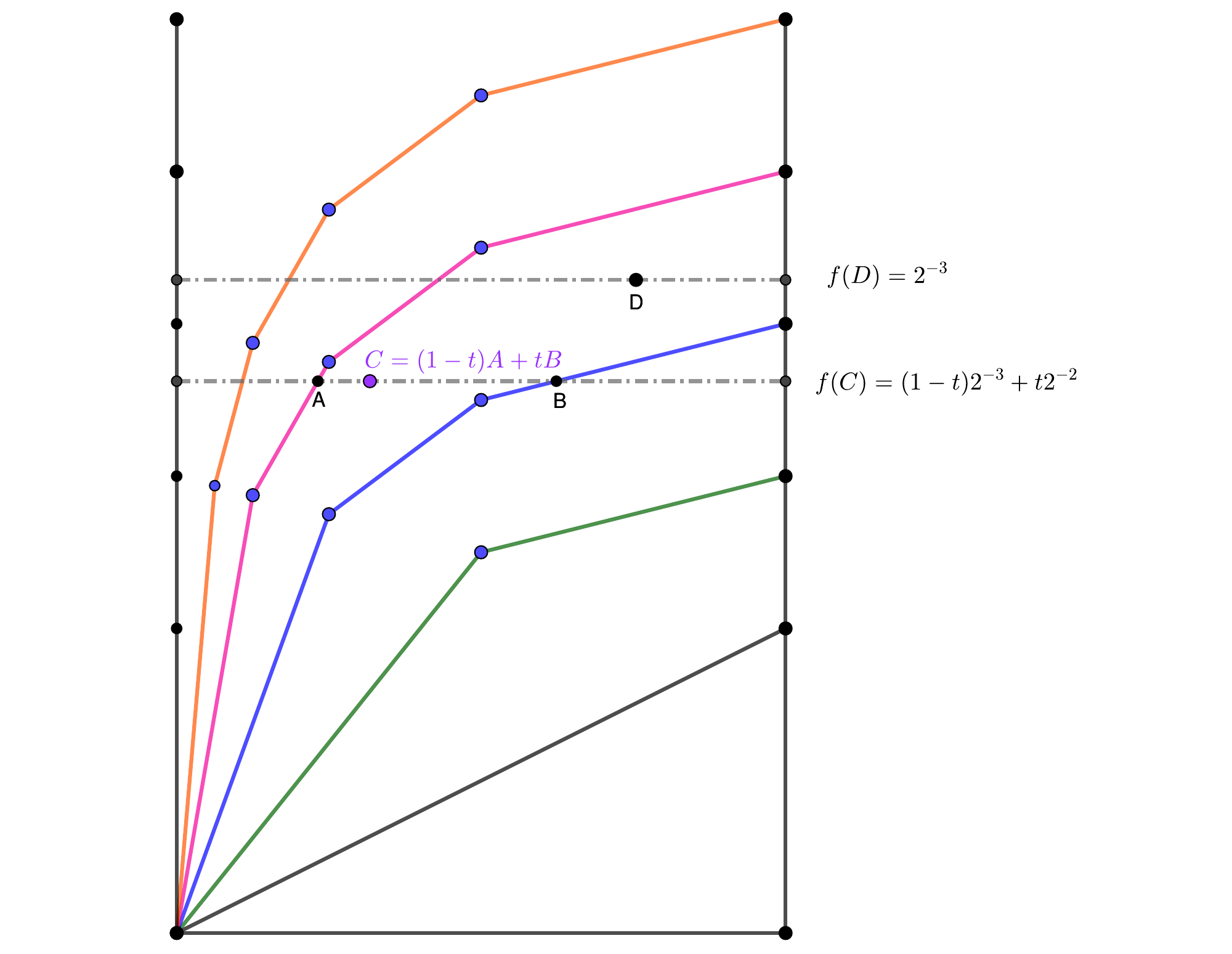}
 \caption{Diagram of $\candidatef$. Here $C$ and $D$ represent the two possible positions of the point $(x_\ast, \lambda_\ast)$.}
\end{figure}

We can now fully state our main theorem:
\begin{TheoremLetter} \label{MainTheorem}
  The function $\candidateB$ constructed above agrees pointwise with $\mathbb{B}$:
  \begin{align*}
    \mathbb{B} = \mathbf{B}
  \end{align*}
  everywhere in the domain $[0,1] \times [0,2] \times \mathbb{R}$.
\end{TheoremLetter}

As we mentioned above, the proof of Theorem \ref{MainTheorem} follows the Bellman function technique.
One of the main tools that this technique provides is a characterization of $\mathbb{B}$ as the minimizer of a certain class of functions.
Thus, to find $\mathbb{B}$ it suffices to find a candidate $\candidateB$ in this class of functions (which already implies that $\mathbb{B} \leq \candidateB$), and then showing that this candidate is best possible, i.e.: $\mathbb{B} \geq \candidateB$.

The minimization problem which defines $\mathbb{B}$ provides a certain inequality
that all majorizing functions ($\mathbb{B}$ included) must satisfy.
We use this inequality (exclusively) to construct $\candidateB$, hence the inequality
$\mathbb{B} \geq \candidateB$ is guaranteed by construction. This is done in section 3.

Then, what remains is to show that the candidate just constructed belongs to the class of functions of which $\mathbb{B}$ is the minimizer. This, which consists of proving the aforementioned strengthened concavity property, is done in section 4.

The problem we solve in this paper is similar to the one addressed in \cite{MR3161110}, where the unrestricted version was solved.
The reader may observe that inequality \eqref{intro:weak-type} is invariant under the homogeneity $(f,\lambda) \mapsto (tf, t\lambda)$,
and this was used in \cite{MR3161110} to effectively reduce the dimensionality of the problem.
This homogeneity is not present in the \emph{restricted} weak-type problem, so diferent ideas are needed here.

Some authors have explored similar problems.
For example A. Os\polhk{e}kowski studied in \cite{MR4203700} the weak-type $(p,p)$ version of this problem when $p > 1$.
See also \cite{MR3733883} where a similar problem is studied, assuming $f \in L\log L$ instead.

In the weighted setting, one can show
\begin{align} \label{intro:weighted}
  \|Tf\|_{L^{1,\infty}(w)} \lesssim [w]_{A_1}\log(e+[w]_{A_\infty}) \|f\|_{L^1(w)},
\end{align}
when $T$ is a Calder\'on-Zygmund operator, or $T=\mathcal{A}_{\mathcal{S}}$. 
See for example \cite{MR2427454} or \cite{MR3455749}.

It was shown in \cite{MR3812865}, using the Bellman function approach, that some power of the logarithm must be present in \eqref{intro:weighted} when $T$ is a Haar shift operator (which is also used as a model for Calder\'on-Zygmund operators).
Later, it was shown in \cite{MR4150264} (see also \cite{MR4282145}) that the logarithmic term in \eqref{intro:weighted} cannot be improved.
This was proved by constructing explicit examples of weights and functions for a related $L^2$ problem, and then using extrapolation to prove the sharpness of \eqref{intro:weighted}.
The question of whether the logarithmic term in \eqref{intro:weighted} is necessary when $f$ is restricted to characteristic functions, however, open.

\section{The Bellman function setting}
\label{S:2}

In this section we set up the Bellman function framework in which we cast the problem of computing \eqref{intro:full_B}.

First, let us introduce the class of all finite Carleson sequences with \emph{Carleson height} equal to $A$:
\begin{align*}
  \mathcal{C}(A) = \{\alpha:\, \|\alpha\|_{\car} \leq 2,\, A(\alpha, [0,1)) = A, \text{ and } \alpha_J = 0 \text{ for all sufficiently small $J$}\}.
\end{align*}
By a standard limiting argument, we have for all $(x,A,\lambda) \in \Omega_{\mathbb{B}} := [0,1] \times [0,2] \times \mathbb{R}$
\begin{align*}
  \mathbb{B}(x,A,\lambda) = \sup\Big\{ |\{t \in [0,1):\, \mathcal{A}_{\mathcal{S}}\ind_E(t) \geq \lambda\}|:\, \alpha \in \mathcal{C}(A),\, E \subseteq [0,1), \text{ and } |E| = x \Big\}.
\end{align*}
It will be convenient to introduce the $V_\lambda$ operator, which acts on pairs $(E,\alpha)$ where $E \subseteq [0,1)$ and $\alpha$ is a Carleson sequence:
\begin{align*}
  V_\lambda(E,\alpha) = |\{t \in [0,1):\, \mathcal{A}_{\mathcal{S}}\ind_E(t) \geq \lambda\}|.
\end{align*}
Then
\begin{align*}
  \mathbb{B}(x,A,\lambda) = \sup\{V_\lambda(E,\alpha):\, |E| = x \text{ and } \alpha \in \mathcal{C}(A)\}
\end{align*}
($E$ is always a subset of $[0,1)$).

The operator $V_\lambda$ behaves well under \emph{concatenation}, which is the operation of taking two scaled copies of a set (or a Carleson sequence) adapted to $[0,1)$ and placing them (appropriately rescaled)
on the two dyadic children of the interval $[0,1)$. In particular for $E_1, E_2 \subseteq [0,1)$ let
\begin{align*}
  E_1 \oplus E_2 = \bigg(\frac{1}{2}E_1\bigg) \cup  \bigg( \frac{1}{2}E_2 + \frac{1}{2} \bigg).
\end{align*}
Similarly, if $\alpha$ and $\beta$ are Carleson sequences we define, for every $\gamma \in [0,1]$
\begin{align*}
  (\alpha \oplus_\gamma \beta)_J := \begin{cases}
    \gamma &\text{if } J = [0,1)\\
    \alpha_{\widehat{J}} &\text{if } J \subseteq [0,2^{-1}) \\
    \beta_{\widehat{J}} &\text{if } J \subseteq [2^{-1}, 1).
  \end{cases}
\end{align*}
here we are denoting the dyadic parent of $J$ by $\widehat{J}$.

The concatenation of two Carleson sequences is depicted graphically in Figure \ref{fig:concat}.
\begin{figure}[H]
  \begin{tikzpicture}[scale=0.5]
    \draw[thick] (0,0) -- (16,0);
    \draw[thick] (0,8) -- (16,8);
    \draw[thick] (0,0) -- (0, 8);
    \draw[thick] (16,0) -- (16, 8);

    \draw[thick] (0,2) -- (16,2);
    \draw[thick] (0,4) -- (16,4);
    \draw[thick] (0,6) -- (16,6);

    \draw[thick] (2 , 0) -- (2 , 2);
    \draw[thick] (8 , 0) -- (8 , 6);
    \draw[thick] (8 , 0) -- (8 , 6);
    \draw[thick] (4 , 0) -- (4 , 4);
    \draw[thick] (12, 0) -- (12, 4);
    \draw[thick] (6 , 0) -- (6 , 2);
    \draw[thick] (10, 0) -- (10, 2);
    \draw[thick] (14, 0) -- (14, 2);

    \node at  (8, 7) {$\gamma$};

    \node at  (4 , 5) {$\alpha_{{I}}$};
    \node at  (12, 5) {$\beta_{{I}}$};

    \node at  (2 , 3) {$\alpha_{{I}_{-}}$};
    \node at  (6 , 3) {$\alpha_{{I}_{+}}$};

    \node at  (10, 3) {$\beta_{{I}_{-}}$};
    \node at  (14, 3) {$\beta_{{I}_{+}}$};

    \node at  (1 , 1) {$\alpha_{{I}_{--}}$};
    \node at  (3 , 1) {$\alpha_{{I}_{-+}}$};
    \node at  (5 , 1) {$\alpha_{{I}_{+-}}$};
    \node at  (7 , 1) {$\alpha_{{I}_{++}}$};
    \node at  (9 , 1) {$\beta_{{I}_{--}}$};
    \node at  (11, 1) {$\beta_{{I}_{-+}}$};
    \node at  (13, 1) {$\beta_{{I}_{+-}}$};
    \node at  (15, 1) {$\beta_{{I}_{++}}$};

    \draw[thick] (0,0) -- (0,-2) -- (16,-2) -- (16, 0);
    \node at  (8, -1) {$\dots$};

    \node at  (-2, 3) {$\alpha \oplus_\gamma \beta = $};
  \end{tikzpicture}
\caption{Concatenation of $\alpha$ and $\beta$, assigning $\gamma$ to the top interval. Here $I = [0,1)$ and the suffix $+$ or $-$ denotes the right or left child respectively.}
  \label{fig:concat}
\end{figure}
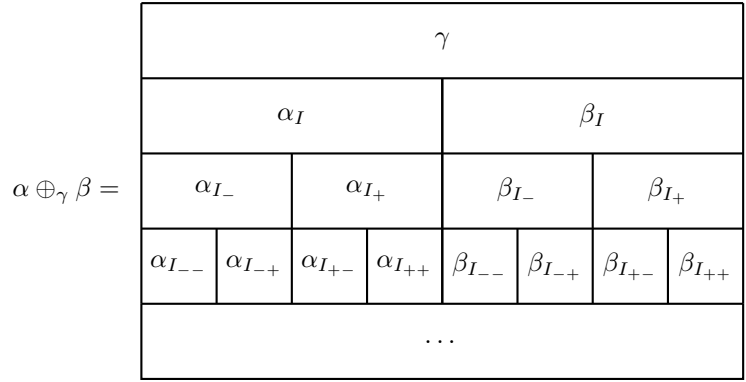

With this notation we can write precisely the dynamics of how the operator $V_\lambda$ behaves under concatenation:
\begin{align} \label{bellman:v_dynamics}
  V_{\lambda + \gamma x}(E_1 \oplus E_2, \alpha_1 \oplus_\gamma \alpha_2) = \frac{1}{2} \big(V_{\lambda}(E_1, \alpha_1) + V_{\lambda}(E_2, \alpha_2)\big),
\end{align}
where $x = |E_1 \oplus E_2|$.

The Carleson height also behaves well under concatenation:
\begin{align*}
  A(\alpha \oplus_\gamma \beta) = \frac{A(\alpha) + A(\beta)}{2} + \gamma.
\end{align*}

We now arrive at the \emph{main inequality} that $\mathbb{B}$ satisfies.
\begin{theorem}
  For any two points $(x_i, A_i) \in [0,1] \times [0,2]$, $i \in \{1,2\}$, denote by $(x, A)$ the mid-point:
  \begin{align*}
    (x,A) = \frac{1}{2}\big((x_1,A_1) + (x_2, A_2)\big).
  \end{align*}
  Then, for all $\lambda \in \mathbb{R}$ and all $\gamma \in \{0,1\}$ such that $A + \gamma \leq 2$ we have
  \begin{align} \label{bellman:mi_gamma}
    \mathbb{B}(x,A+\gamma,\lambda + \gamma x) \geq \frac{1}{2}\sum_{i=1}^2 \mathbb{B}(x_i, A_i, \lambda)
  \end{align}
\end{theorem}
\begin{proof}
  Fix $\varepsilon > 0$ and let $(E_i, \alpha_i)$ be almost-extremizers in the definition of $\mathbb{B}(x_i,A_i,\lambda)$, i.e.:
  \begin{align*}
    \mathbb{B}(x_i,A_i,\lambda) \leq V_\lambda(E_i, \alpha_i) + \varepsilon \qquad \text{for i } \in \{1,2\}.
  \end{align*}
  We can form the concatenation of these two examples: $E = E_1 \oplus E_2$ and $\alpha = \alpha_1 \oplus_\gamma \alpha_2$.
  It is easy to see that
  \begin{align*}
    |E| = \frac{x_1+x_2}{2} \quad\text{and}\quad A(\alpha) = \frac{A_1 + A_2}{2} + \gamma.
  \end{align*}

  We now have $(x,A) \in [0,1] \times [0,2]$ and $\alpha \in \mathcal{C}(A)$, by our assumptions, so
  \begin{align*}
    \mathbb{B}(x,A + \gamma,\lambda + \gamma x) &\geq V_{\lambda + \gamma x}(E, \alpha) \\
    &= \frac{1}{2}\sum_{i=1}^2 V_\lambda(E_i, \alpha_i) \tag{by \eqref{bellman:v_dynamics}} \\
    &\geq \frac{1}{2}\sum_{i=1}^2 \mathbb{B}(x_i, A_i, \lambda) - \varepsilon.
  \end{align*}
  The claim follows after taking $\varepsilon \to 0$.
\end{proof}

Inequality \eqref{bellman:mi_gamma} can be seen as a strengthened concavity property, in fact when $\gamma = 0$ it says precisely that $\mathbb{B}(\cdot, \cdot, \lambda)$ is mid-point concave.

The function $\mathbb{B}$ also satisfies the following \emph{obstacle condition}
\begin{align} \label{bellman:obstacle_prev}
  \mathbb{B}(x,A,\lambda) = 1 \qquad \text{whenever }\lambda \leq 0,
\end{align}
which is just the observation that the operator $\mathcal{A}$ is positive.

It turns out that $\mathbb{B}$ is the minimizer among all the functions satisfying \eqref{bellman:mi_gamma} and \eqref{bellman:obstacle_prev}.
Before stating the precise version of this minimization property, let us introduce the following definitions
\begin{definition}
  We will say that a function $B: \Omega_{\mathbb{B}} \to \mathbb{R}_{\geq 0}$ \textbf{satisfies the main inequality} if for both $\gamma \in \{0,1\}$
  \begin{align} \label{bellman:mi}
    B(x,A+\gamma, \lambda + \gamma x) \geq \frac{1}{2}\sum_{i=1}^2 B(x_i,A_i,\lambda),
  \end{align}
  whenever $(x_i,A_i,\lambda) \in \Omega_{\mathbb{B}}$, $(x,A+\gamma,\lambda + \gamma x) \in \Omega_{\mathbb{B}}$, where
  \begin{align*}
    x = \frac{x_1+x_2}{2} \qquad\text{and}\qquad A = \frac{A_1+A_2}{2}.
  \end{align*}
\end{definition}
\begin{definition}
  We will say that a function $B: \Omega_{\mathbb{B}} \to \mathbb{R}_{\geq 0}$ \textbf{satisfies the obstacle condition} if
  \begin{align} \label{bellman:obstacle}
    B(x,A,\lambda) = 1
  \end{align}
  whenever $\lambda \leq 0$.
\end{definition}
Finally, we define the following class of functions:
\begin{align*}
  \mathcal{B} = \{\text{functions }B : \Omega_{\mathbb{B}} \to \mathbb{R}_{\geq 0} \text{ satisfying the main inequality and the obstacle condition}\}.
\end{align*}
So far, we have shown that $\mathbb{B} \in \mathcal{B}$. The next theorem makes the fact that $\mathbb{B}$ is a minimizer of $\mathcal{B}$ precise
\begin{theorem} \label{bellman:FTB}
  For every $B \in \mathcal{B}$ we have $\mathbb{B} \leq B$.
\end{theorem}
\begin{proof}
  Let $E \subseteq [0,1)$ and $\alpha$ be any finite Carleson sequence. Letting $I = [0,1)$ and $I_{\pm}$ be the left and right dyadic children of $I$:
  \begin{align*}
    I_- = [0, 2^{-1}) \qquad\text{and}\qquad I_+ = [2^{-1}, 1)
  \end{align*}
  we can decompose
  \begin{align*}
    E &= E_{|I_-} \oplus E_{|I_+} \\
    \alpha &= \alpha_{|I_-} \oplus_{\alpha_I} \alpha_{|I_+}.
  \end{align*}
  Here we are denoting by $E_{|J}$ and $\alpha_{|J}$ the restrictions of $E$ and $\alpha$ respectively to an interval $J$.

  Using \eqref{bellman:mi}:
  \begin{align}
    B(|E|, A(\alpha), \lambda) &= B(|E_{|I_-} \oplus E_{|I_+}|, A(\alpha_{|I_-} \oplus_{\alpha_I} \alpha_{|I_+}), \lambda) \notag \\
    &= B\bigg(\frac{|E_{|I_-}| + |E_{|I_+}|}{2}, \frac{A(\alpha_{|I_-}) + A(\alpha_{|I_+})}{2} + \alpha_I, \lambda\bigg) \notag \\
    &\geq \frac{1}{2}\sum_{J \in \mathcal{D}_1(I)} B(|E_{|J}|, A(\alpha_{|J}), \lambda - |E|\alpha_I) \label{bellman:min:ind},
  \end{align}
  where we are denoting by $\mathcal{D}_m$ the $m$-th generation dyadic children of $I$, which we can define inductively by $\mathcal{D}_0(I) = \{I\}$ and
  \begin{align*}
    \mathcal{D}_{m+1}(I) = \{J_{-}:\, J \in \mathcal{D}_m(I)\} \cup \{J_{+}:\, J \in \mathcal{D}_m(I)\}
  \end{align*}
  for $m \geq 1$.

  We can iterate \eqref{bellman:min:ind} to arrive at
  \begin{align*}
    B(|E|, A(\alpha), \lambda) &\geq \frac{1}{2^2}\sum_{J \in \mathcal{D}_2(I)} B(|E_{|J}|, A(\alpha_{|J}), \lambda - |E|\alpha_I - |E_{|J^{(1)}}|\alpha_{J^{(1)}}) \\
    &\qquad \vdots \\
    &\geq \frac{1}{2^N} \sum_{J \in \mathcal{D}_N(I)} B\bigg(|E_{|J}|, A(\alpha_{|J})), \lambda - \sum_{m=1}^{N} |E_{|J^{(m)}}| \alpha_{J^{(m)}} \bigg),
  \end{align*}
  where we are letting $J^{(m)}$ denote the $m$-th dyadic parent of $J$.

  Note that we have reconstructed the operator $\mathcal{A}_{\alpha}(\ind_E)$ in the quantity being subtracted from $\lambda$. (Here we denoted $\mathcal{A}_{\alpha}=\mathcal{A}_{\mathcal{S}}$, where $\mathcal{S}$ is the sparse collection associated with the Carleson sequence $\alpha$.)
  In fact, if $\alpha_J = 0$ for all $J$ of generation $N$ or greater, then the function $\mathcal{A}_{\alpha}(\ind_E)$ is constant on every interval $J \in \mathcal{D}_N(I)$ and we have
  \begin{align*}
    B(|E|, A(\alpha), \lambda) &\geq \int_0^1 B\bigg(|E_{|J}|, A(\alpha_{|J})), \lambda - \mathcal{A}_{\alpha}(\ind_E)(t) \bigg) \, dt.
  \end{align*}

  Consider the level set $L_\lambda = \{t \in [0,1):\, \mathcal{A}_\alpha(\ind_E)(t) \geq \lambda\}$.
  We can use the obstacle condition \eqref{bellman:obstacle} to arrive at
  \begin{align*}
    B(|E|, A(\alpha), \lambda) &\geq \int_{L_\lambda} B\bigg(|E_{|J}|, A(\alpha_{|J})), \lambda - \mathcal{A}_{\alpha}(\ind_E)(t) \bigg) \, dt \\
    &= \int_{L_\lambda} 1 \, dt = |L_\lambda|.
  \end{align*}
  Thus, we have shown that
  \begin{align*}
    |\{t \in [0,1):\, \mathcal{A}_\alpha(\ind_E)(t) \geq \lambda\}| \leq B(|E|, A(\alpha), \lambda),
  \end{align*}
  whenever $\alpha$ is a finite Carleson sequence. A simple limiting argument can be used to remove the finiteness assumption.

  Taking supremums as in the definition of $\mathbb{B}$ yields
  \begin{align*}
    \mathbb{B}(x,A,\lambda) \leq B(x,A,\lambda)
  \end{align*}
  for all $(x, A, \lambda) \in \Omega_{\mathbb{B}}$, as claimed.
\end{proof}

The strategy in the rest of the paper is to use the main inequality and the obstacle condition to obtain
progressively stronger lower bounds for all functions in $\mathcal{B}$.
These lower bounds will then form our candidate $\candidateB$, thus we will have $\mathbb{B} \geq \candidateB$ by construction.
Then, the candidate will be shown to equal $\mathbb{B}$ by proving that it satisfies the obstacle condition and the main inequality, and then applying Theorem \ref{bellman:FTB}.

Before proceeding with the construction of $\candidateB$ let us note two remarks which will both guide the construction, and simplify the verification of the main inequality.

First, the main inequality naturally splits into two properties corresponding to setting $\gamma = 0$ or $\gamma = 1$. Let $B$ be any function in $\mathcal{B}$,
when $\gamma = 0$ the main inequality states that
\begin{align*}
  B\bigg(\frac{x_1+x_2}{2}, \frac{A_1+A_2}{2}, \lambda\bigg) \geq \frac{1}{2}\sum_{i=1}^2 B(x_i,A_i,\lambda)
\end{align*}
whenever $(x_i,A_i,\lambda) \in \Omega_{\mathbb{B}}$ for $i=1,2$.
Note that the domain $\Omega_{\mathbb{B}}$ is convex, so this condition is enough to guarantee that the left-hand side is well-defined.
In other words, when $\gamma = 0$ the main inequality says that $B$  is mid-point concave.

When $\gamma = 1$ we can set $x_1 = x_2 = x$ and $A_1 = A_2 = A$ in \eqref{bellman:mi} to obtain
\begin{align*}
  B(x, A+1, \lambda + x) \geq B(x,A,\lambda)
\end{align*}
whenever $x \in [0,1]$ and $0 \leq A \leq 1$.

These two properties, together with the obstacle condition, are actually sufficient for $B$ to belong to $\mathcal{B}$:
\begin{theorem}
  Suppose $B : \Omega_{\mathbb{B}} \to \mathbb{R}_{\geq 0}$ satisfies the obstacle condition \eqref{bellman:obstacle} and
  \begin{align}
    \tag{Concavity} B(\cdot, \cdot, \lambda) \text{ is mid-point concave} \label{bellman:concavity} \\
    \tag{J} B(x, A+1, \lambda + x) \geq B(x,A,\lambda) \label{bellman:jump}
  \end{align}
  whenever $x \in [0,1]$ and $0 \leq A \leq 1$. Then $B \in \mathcal{B}$.
\end{theorem}
\begin{proof}
  We just need to check \eqref{bellman:mi} with $\gamma = 1$, so let $(x_i, A_i, \lambda) \in \Omega_{\mathbb{B}}$, set $x = (x_1 + x_2)/2$ and $A = (A_1 + A_2)/2$,
  and assume $0 \leq A \leq 1$. Using \eqref{bellman:jump} we have
  \begin{align*}
    B(x,A+1,\lambda+x) &\geq B(x,A,\lambda).
  \end{align*}
  Now, using \eqref{bellman:concavity}:
  \begin{align*}
    B(x,A,\lambda) \geq \frac{1}{2}\sum_{i=1}^2 B(x_i,A_i, \lambda)
  \end{align*}
  which yields the claim.
\end{proof}

Note that the function $B \equiv 1$ trivially belongs to $\mathcal{B}$, so we can restrict our search to functions taking values only in $[0,1]$.
This boundedness allows us to also upgrade mid-point concavity to full concavity:
\begin{proposition}
  Let $B \in \mathcal{B}$ be any function bounded by $1$, then
  $B(\cdot, \cdot, \lambda)$ is continuous and concave for every $\lambda$.
\end{proposition}
This theorem is a consequence of the fact that bounded mid-point concave functions are continuous (and hence concave), see page 12 in \cite{MR3363413}.

Without loss of generality, we assume from now on that all functions in $\mcb$ are $[0,1]$-valued.

\section{Constructing the candidate $\bfb$}
\label{S:3}

In this section we construct the candidate Bellman function $\candidateB$.
This will be done by first giving lower bounds for $\mathbb{B}$ along a certain collection of curves,
and then using concavity to fill-in the rest of the domain.

We begin by using the information given by the obstacle condition \eqref{bellman:obstacle}.
We can use inequality \eqref{bellman:jump}, which we will call \emph{the jump inequality},
to ``jump'' a lower bound for $\mathbb{B}(x,A,\lambda)$ to a lower bound for $\mathbb{B}(x, A+1, \lambda + x)$.
\begin{proposition} \label{construction:first_nontrivial}
  For every $0 \leq t \leq x \leq 1$ we have
  \begin{align*}
    \mathbb{B}(x, 1, t) &= 1, \\
    \mathbb{B}(x, 2, 2t) &= 1.
  \end{align*}
\end{proposition}
\begin{proof}
  Since $t-x \leq 0$, by \eqref{bellman:obstacle}
  \begin{align*}
    \mathbb{B}(x, 0, t-x) = 1 \stackrel{\eqref{bellman:jump}}{\implies} \mathbb{B}(x,1,t) \geq 1
  \end{align*}

  If we instead start from $2(t-x)$, we can apply \eqref{bellman:jump} twice:
  \begin{align*}
    1 \stackrel{\eqref{bellman:obstacle}}{=} \mathbb{B}(x, 0, 2(t-x)) \stackrel{\eqref{bellman:jump}}{\leq} \mathbb{B}(x, 1, 2t - x) \stackrel{\eqref{bellman:jump}}{\leq} \mathbb{B}(x, 2, 2t).
  \end{align*}
  Recall that the constant function $1$ is in $\mathcal{B}$, so $\mathbb{B} \leq 1$ which provides the equality.
\end{proof}

This proposition gives us the first non-trivial lower bound for $\mathbb{B}$.
Unfortunately, the idea of jumping repeatedly cannot be used alone to obtain more lower bounds, as every time \eqref{bellman:jump} is used the variable $A$ increases by $1$
and we quickly find ourselves outside of the domain.

Keeping in mind that we want to find the smallest concave function satsifying a certain inequality (namely the jump inequality \eqref{bellman:jump}), 
it makes sense to first obtain lower bounds on the parts of the domain where
it is harder to use concavity (i.e.: extreme points, boundary, etc.),
and then constructing a concave function on top of these lower bounds.

With this guiding principle, we will first obtain lower bounds for $\mathbb{B}(1, 2, \lambda)$ and then for $\mathbb{B}(x,2,\lambda)$.
These boundary lower bounds will then be extended to the rest of the domain using concavity.
Since we will be using the restriction of $\mathbb{B}$ to the $A=2$ boundary, we will give the restriction a name:
\begin{align*}
  \mathbb{F}(x,\lambda) = \mathbb{B}(x, 2, \lambda).
\end{align*}
Proposition \ref{construction:first_nontrivial} gives $\mathbb{F}(1, \lambda) = 1$ for $\lambda \leq 2$.
In order to obtain values for larger $\lambda$ we need to use \eqref{bellman:jump}, but to do so we first need to reduce the $A$ variable.
Observe that the function $t \mapsto \mathbb{B}(1, 2t, \lambda)$ is concave, so
\begin{align*}
  \mathbb{B}(1, 1, \lambda) &\geq \frac{1}{2}\mathbb{B}(1, 0, \lambda) + \frac{1}{2}\mathbb{B}(1, 2, \lambda) \\
  &\geq \frac{1}{2}\mathbb{B}(1, 2, \lambda),
\end{align*}
since $\mathbb{B} \geq 0$.
We can now apply \eqref{bellman:jump} to obtain
\begin{align} \label{construction:x=1}
  \mathbb{F}(1, \lambda + 1) \geq \frac{1}{2}\mathbb{F}(1, \lambda).
\end{align}
We thus have
\begin{align*}
  \mathbb{F}(1,\lambda) \geq \begin{cases}
    1 &\text{for all }\lambda \leq 2, \\
    2^{-1} &\text{for all }\lambda \in (2, 3].
  \end{cases}
\end{align*}
Inductively applying \eqref{construction:x=1} we arrive at
\begin{align} \label{construction:x=1:full}
  \mathbb{F}(1,\lambda) \geq \begin{cases}
    1 &\text{for all }\lambda \leq 2, \\
    2^{-m} &\text{for all }\lambda \in (m+1, m+2] \text{ and every integer }m \geq 1.
  \end{cases}
\end{align}

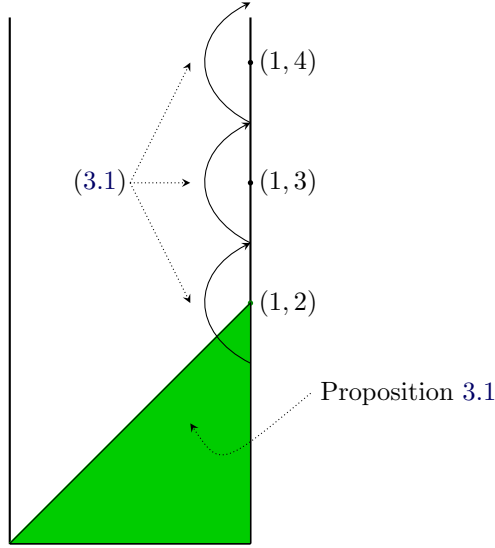
\begin{figure}[H]
  \begin{tikzpicture}[scale=0.2]
    \draw[thick, -] (0, 0) -- (16, 0);
    \draw[thick, -] (16, 0) -- (16, 35);
    \draw[thick, -] (0, 0) -- (0, 35);

    \draw[thick, -, green1] (0, 0) -- (16, 16);

    \node at  (16+2.5, 16) {$(1,2)$};
    \fill[green2] (16, 16) circle (5pt);

    \node at  (16+2.5, 24) {$(1,3)$};
    \fill (16, 24) circle (5pt);

    \node at  (16+2.5, 32) {$(1,4)$};
    \fill (16, 32) circle (5pt);

    \filldraw[fill=green1] (0,0) -- (16, 0) -- (16, 16) -- cycle;
    \draw[densely dotted, -stealth] (20, 10) .. controls (14,5) .. (12, 8);
    \node at  (20+6.5, 10) {Proposition \ref{construction:first_nontrivial}};

    \draw[-stealth] (16, 12) .. controls (12,14) and (12, 18) .. (16, 20);
    \draw[-stealth] (16, 20) .. controls (12,22) and (12, 26) .. (16, 28);
    \draw[-stealth] (16, 28) .. controls (12,30) and (12, 34) .. (16, 36);

    \node at  (6, 24) {\eqref{construction:x=1}};
    \draw[densely dotted, -stealth] (8, 24) to (12, 24);
    \draw[densely dotted, -stealth] (8, 24) to (12, 32);
    \draw[densely dotted, -stealth] (8, 24) to (12, 16);

  \end{tikzpicture}
\caption{Inductive application of \eqref{construction:x=1} to obtain \eqref{construction:x=1:full}.}
  \label{fig:x=1}
\end{figure}

This gives us (actually sharp) lower bounds for $\mathbb{F}(1, \cdot)$. In order to obtain lower bounds for different values of $x$ we need a different approach.
If we want to apply \eqref{bellman:jump} to obtain information about $\mathbb{F}$, we need to apply it to points with $A=1$. Thus, let us introduce the companion to $\mathbb{F}$:
\begin{align*}
  \mathbb{G}(x, \lambda) = \mathbb{B}(x,1,\lambda).
\end{align*}

The jump inequality \eqref{bellman:jump} allows us to transfer lower bounds for $\mathbb{G}$ to lower bounds for $\mathbb{F}$.
We still need to gain information about $\mathbb{G}$. To do this, we can use concavity.
In particular, note that we have
\begin{align} \label{construction:C} \tag{C}
  \mathbb{G}\Big(\frac{x}{2}, \lambda\Big) \geq \frac{1}{2}\mathbb{F}(x, \lambda),
\end{align}
which is just the observation that the function $t \mapsto \mathbb{B}(tx, 2t, \lambda)$ is concave on $[0,1]$.

Thus, we have a way to go from $\mathbb{G}$ to $\mathbb{F}$ \eqref{bellman:jump}, and from $\mathbb{F}$ to $\mathbb{G}$ \eqref{construction:C}.
We will use these two repeatedly, so let us introduce the mappings
\begin{align*}
  J(x,\lambda) = (x, \lambda + x) \qquad \text{and} \qquad C(x,\lambda) = \Big(\frac{x}{2}, \lambda\Big).
\end{align*}
Then
\begin{align*}
  \mathbb{F} \circ J \geq \mathbb{G} \qquad \text{and} \qquad \mathbb{G} \circ C \geq \frac{1}{2}\mathbb{F}.
\end{align*}
Let us apply these operations to the lower bounds given by Proposition \ref{construction:first_nontrivial}.

Proposition \ref{construction:first_nontrivial} gives us the bound
\begin{align*}
  \mathbb{F}(x, 2x) \geq 1.
\end{align*}
After applying \eqref{construction:C}:
\begin{align*}
  \mathbb{G}\Big(\frac{x}{2}, 2x\Big) \geq \frac{1}{2},
\end{align*}
and now applying \eqref{bellman:jump}:
\begin{align*}
  \mathbb{F}\Big(\frac{x}{2}, \frac{5x}{2}\Big) \geq \frac{1}{2}.
\end{align*}
We now have information about $\mathbb{F}$ on two lines, as depicted in Figure \ref{first_two}.
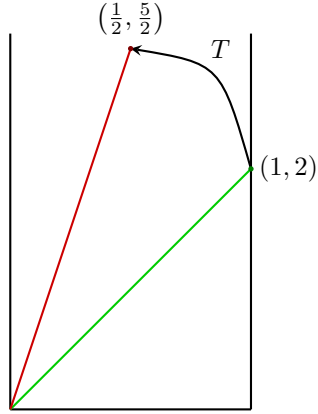
\begin{figure}[H]
  \begin{tikzpicture}[scale=0.2]
    \draw[thick, -] (0, 0) -- (16, 0);
    \draw[thick, -] (16, 0) -- (16, 25);
    \draw[thick, -] (0, 0) -- (0, 25);

    \draw[thick, -, green1] (0, 0) -- (16, 16);
    \draw[thick, -, red1] (0, 0) -- (8, 24);

    \draw[thick, -stealth] (16, 16) .. controls (14,23) .. (8, 24);

    \node at  (16+2.5, 16) {$(1,2)$};
    \node at  (8, 24 + 2) {$\big(\frac{1}{2},\frac{5}{2}\big)$};
    \node at  (14, 24) {$T$};

    \fill[green2] (16, 16) circle (5pt);
    \fill[red2] (8, 24) circle (5pt);
  \end{tikzpicture}
\caption{Lower bounds for $\mathbb{F}$. The green line is given by Proposition \ref{construction:first_nontrivial}, and the red is obtained by applying \eqref{construction:C} and \eqref{bellman:jump}.}
  \label{first_two}
\end{figure}

Since we will be applying $J$ and $C$ repeatedly, let us introduce
\begin{align*}
  T(x,\lambda) &= (J\circ C)(x,\lambda) = \Big(\frac{x}{2}, \lambda + \frac{x}{2}\Big).
\end{align*}

This is a linear mapping taking $[0,1] \times \mathbb{R}$ to itself.
We can interpret $T$ geometrically as taking a point $p = (x,\lambda)$ to the midpoint of $p$ and the intercept of the line through $p$ of slope $-1$ and the $x=0$ axis.
Since it is linear, it maps lines to lines. In particular, if $\ell$ is a line of slope $m$ then $T(\ell)$ is a line of slope $2m +1$.
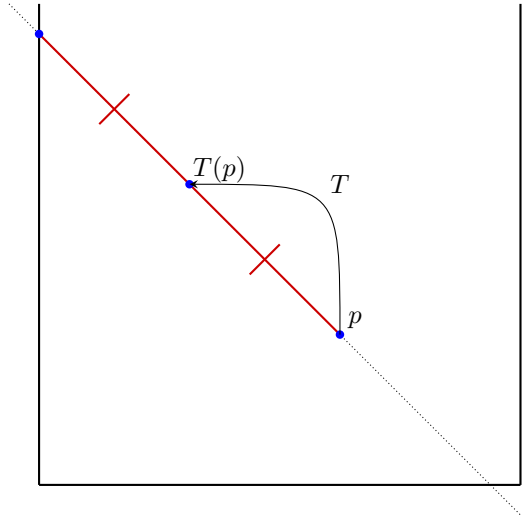
\begin{figure}[H]
  \begin{tikzpicture}[scale=0.2]
    \draw[thick, -] (0, 0) -- (32, 0);
    \draw[thick, -] (32, 0) -- (32, 32);
    \draw[thick, -] (0, 0) -- (0, 32);

    \draw[densely dotted] (-2, 32) -- (32, -2);

    \draw[red1, thick, -] (0, 30) -- (10, 20);
    \draw[red1, thick, -] (10, 20) -- (20, 10);

    \fill[blue] (20, 10) circle (8pt);
    \fill[blue] (0, 30) circle  (8pt);
    \fill[blue] (10, 20) circle (8pt);

    \draw[red1, thick, -] (4, 24) -- (6, 26);
    \draw[red1, thick, -] (14, 14) -- (16, 16);

    \draw[-stealth] (20, 10) .. controls (20, 20) .. (10, 20);

    \node at  (20, 20) {$T$};

    \node at  (21, 11) {$p$};
    \node at  (12, 21) {$T(p)$};
  \end{tikzpicture}
  \caption{Geometric interpretation of $T$.}
  \label{fig:geometric_interpretation_of_T}
\end{figure}

This map can be used to propagate information about $\mathbb{F}$ over its domain, specifically we can use the inequality
\begin{align} \label{construction:T}
  \mathbb{F}(T(x,\lambda)) \geq \frac{1}{2}\mathbb{F}(x,\lambda).
\end{align}
If we apply this inequality repeatedly we arrive the following proposition.
\begin{proposition} \label{construction:first_fan}
  Define the points $F_k$ for $k \geq 0$ as $F_0 = (1,2)$ and
  \begin{align*}
    F_k = \Big(\frac{1}{2^k}, 3 - \frac{1}{2^k} \Big) \quad \text{for }k \geq 1.
  \end{align*}
  Then
  \begin{align*}
    \mathbb{F} \geq 2^{-k} \qquad \text{on the line joining $(0,0)$ and $F_k$}.
  \end{align*}
\end{proposition}
\begin{proof}
  The previous discussion established the inequality for $k=0$ and $k=1$. Assume by induction that it holds for some $k \geq 1$, then
  since the mapping $T$ takes lines to lines, we just need to check that $T(F_k) = F_{k+1}$. Indeed
  \begin{align*}
    T(F_k) &= T\Big(\frac{1}{2^k}, 3 - \frac{1}{2^k} \Big) \\
    &= \Big( \frac{1}{2} \cdot \frac{1}{2^k}, 3 - \frac{1}{2^k} + \frac{1}{2} \cdot \frac{1}{2^k}\Big) \\
    &= \Big(\frac{1}{2^{k+1}}, 3 - \frac{1}{2^{k+1}} \Big) = F_{k+1}.
  \end{align*}
\end{proof}

Let us recapitulate what we have proven so far, and also explain how we can use concavity to fill-in lower bounds for $\mathbb{B}$ in the rest of the domain.
We will use this opportunity to give a complete lower bound for $\mathbb{B}(\cdot, \cdot, \lambda)$ for all $\lambda \in [0,1]$.

Fix $\lambda \in [0,1]$ and consider the two-variable function $\mathbb{B}(\cdot, \cdot, \lambda)$.
Proposition \ref{construction:first_nontrivial} gives us
\begin{align*}
  \mathbb{B}(x, 1, \lambda) &= 1, \\
  \mathbb{B}\bigg(\frac{x}{2}, 2, \lambda\bigg) &= 1,
\end{align*}
for all $x \geq \lambda$. The convex hull of the set given by
\begin{align*}
  \{(x,1):\, x \in [\lambda, 1]\} \cup \bigg\{\bigg( \frac{x}{2}, 2\bigg):\, x \in [\lambda, 1] \bigg\}
\end{align*}
is the trapezoid with vertices $\{(\lambda, 1), (1,1), (1, 2), (\frac{\lambda}{2},2)\}$. This is depicted in Figure \ref{fig:lambda_leq_1}
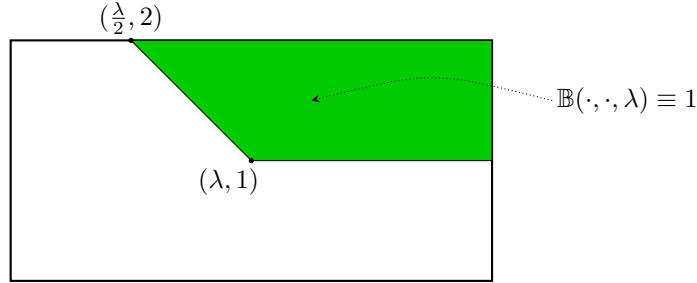
\begin{figure}[H]
  \begin{tikzpicture}[scale=0.2]
    \draw[thick, -] (0, 0) -- (32, 0) -- (32, 16) -- (0, 16) -- cycle;

    \filldraw[fill=green1] (16,8) -- (32, 8) -- (32, 16) -- (8, 16) -- cycle;

    \draw[densely dotted, -stealth] (36, 12) .. controls (28,14) .. (20, 12);
    \node at  (36+5, 12) {$\mathbb{B}(\cdot, \cdot, \lambda) \equiv 1$};

    \fill (8, 16) circle (5pt);
    \node at  (8, 16+1.5) {$(\frac{\lambda}{2},2)$};
    \fill (16, 8) circle (5pt);
    \node at  (16-1.5, 8-1.35) {$(\lambda,1)$};
  \end{tikzpicture}
  \caption{Lower bounds for $\mathbb{B}(\cdot, \cdot, \lambda)$ given by Proposition \ref{construction:first_nontrivial} and concavity.}
  \label{fig:lambda_leq_1}
\end{figure}
One can use this, and concavity again, to complete the lower bounds for $\mathbb{B}(\cdot, \cdot, \lambda)$ in the region of the rectangle $[0,1] \times [0,2]$
lying below the line $A = \frac{4x}{\lambda}$.
Inside the triangle given by the vertices $\{(0,0), (\lambda, 1), (\frac{\lambda}{2}, 2)\}$, we can write any point $(x,\lambda)$ as the convex combination of $(0,0)$ and
a point on the line from $(\lambda, 1)$ to $(\frac{\lambda}{2}, 2)$ (where $\mathbb{B} = 1$), therefore we have
\begin{align*}
  \mathbb{B}(x,A,\lambda) \geq \frac{A + \frac{2x}{\lambda}}{3}
\end{align*}
for all $(x,A) \in \polygon\big((0,0), (\lambda, 1), (\frac{\lambda}{2}, 2)\big)$. A simpler computation shows
\begin{align*}
  \mathbb{B}(x,A,\lambda) \geq A
\end{align*}
for all $(x,A) \in \polygon((0,0), (1,0), (1, 1), (\lambda, 1))$.

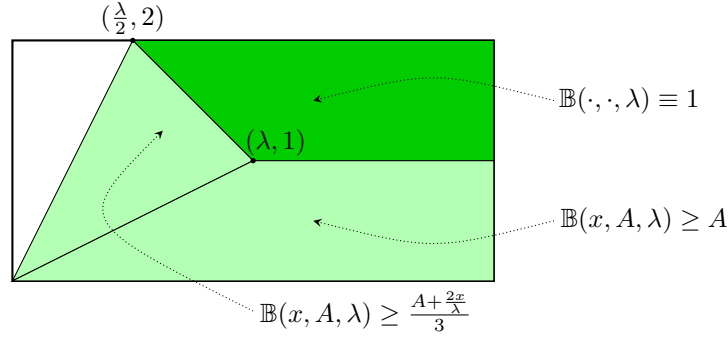
\begin{figure}[H]
  \begin{tikzpicture}[scale=0.2]
    \draw[thick, -] (0, 0) -- (32, 0) -- (32, 16) -- (0, 16) -- cycle;

    \filldraw[fill=green1] (16,8) -- (32, 8) -- (32, 16) -- (8, 16) -- cycle;

    \filldraw[fill=green0] (0,0) -- (16, 8) -- (8, 16) -- cycle;
    \filldraw[fill=green0] (0,0) -- (32, 0) -- (32, 8) -- (16, 8) -- cycle;

    \draw[densely dotted, -stealth] (36, 12) .. controls (28,14) .. (20, 12);
    \node at  (36+5, 12) {$\mathbb{B}(\cdot, \cdot, \lambda) \equiv 1$};

    \draw[densely dotted, -stealth] (36, 4) .. controls (28,2) .. (20, 4);
    \node at  (36+6, 4) {$\mathbb{B}(x, A, \lambda) \geq A$};

    \draw[densely dotted, -stealth] (16, -2) .. controls (4,4) .. (10, 10);
    \node at  (16+7.5, -2) {$\mathbb{B}(x,A,\lambda) \geq \frac{A + \frac{2x}{\lambda}}{3}$};

    \fill (8, 16) circle (5pt);
    \node at  (8, 16+1.5) {$(\frac{\lambda}{2},2)$};
    \fill (16, 8) circle (5pt);
    \node at  (16+1.5, 8+1.25) {$(\lambda,1)$};
  \end{tikzpicture}
  \caption{Further uses of concavity}
  \label{fig:lambda_leq_1:2}
\end{figure}

To fill the remaining part, of the domain, we can use concavity once more. First note that
\begin{align*}
  \mathbb{B}(tx,tA,\lambda) \geq t\mathbb{B}(x,A,\lambda).
\end{align*}
When $A=2$ and $x \leq \frac{\lambda}{2}$ this gives us lower bounds for $\mathbb{B}$ on the remaining part of the domain, provided we can estimate $\mathbb{F}$.
In particular
\begin{align*}
  \mathbb{B}(x,A,\lambda) \geq \frac{A}{2}\mathbb{B}\Big(\frac{2x}{A}, 2, \lambda\Big) = \mathbb{F}\Big(\frac{2x}{A}, \lambda \Big).
\end{align*}

Note that Proposition \ref{construction:first_fan} gives us
\begin{align*}
  \mathbb{F}\bigg(\frac{\lambda}{3\cdot 2^k - 1}, \lambda\bigg) \geq \frac{1}{2^k}.
\end{align*}
Thus, by concavity
\begin{align} \label{construction:lerp1}
  \mathbb{F}(x,\lambda) \geq \iLint[\{(x_k(\lambda), 2^{-k}):\, k \geq 0\}](x),
\end{align}
where
\begin{align*}
  x_k(\lambda) = \frac{\lambda}{3\cdot 2^k - 1}.
\end{align*}
Here we are using $\iLint[\{p_1, p_2, \dots\}]$ to denote the linear interpolation of the points $\{p_1, p_2, \dots\}$, in particular
when $x_1 \leq x_2$, $\iLint$ is given by the point-slope formula
\begin{align*}
  \iLint\big[(x_1, y_1), (x_2, y_2)\big](x) = y_1 + \frac{y_2-y_1}{x_2-x_1}(x-x_1) \qquad x \in [x_1, x_2].
\end{align*}
We can now define our candidate function $\candidateB(\cdot, \cdot, \lambda)$ for $\lambda \in (0,1]$ as
\begin{align*}
  \candidateB(x,A,\lambda) := \begin{cases}
    1 &\text{for $(x,A)$ in } \Sigma_0'(\lambda) \\
    A &\text{for $(x,A)$ in } \Sigma_1(\lambda) \\
    \frac{A+\frac{2x}{\lambda}}{3} &\text{for $(x,A)$ in } \Sigma_0(\lambda) \\
    \frac{A}{2}\candidatef\Big(\frac{2x}{A}, \lambda\Big) &\text{for $(x,A)$ in } \Sigma_2(\lambda) 
  \end{cases}
\end{align*}
where subdomains are the same as in the introduction:
\begin{align*}
  \Sigma_0'(\lambda) &= \polygon\big((\tfrac{\lambda}2, 2), (\lambda, 1), (1,1), (1,2)\big) \\
  \Sigma_0(\lambda) &= \polygon\big((0,0), (\lambda, 1), (\tfrac{\lambda}2, 2)\big) \\
   \Sigma_1(\lambda) &= \polygon\big((0,0), (1,0), (1,1), (\lambda, 1)\big) \\
  \Sigma_2(\lambda) &= \polygon\big((0,0), (\tfrac{\lambda}2, 2), (0,2)\big),
\end{align*}
and
\begin{align} \label{construction:flerp}
  \candidatef(x,\lambda) := \iLint[\{(x_k(\lambda), 2^{-k}):\, k \geq 0\}](x)
\end{align}
for $\lambda \leq 1$.

With this definition we have
\begin{theorem} \label{construction:lambda_leq_1}
  For every $\lambda \in (0, 1]$ and $(x,A) \in [0,1] \times [0,2]$
  \begin{align*}
    \mathbb{B}(x,A,\lambda) \geq \candidateB(x,A,\lambda).
  \end{align*}
\end{theorem}
It remains to provide lower bounds for $\mathbb{B}$ when $\lambda > 1$.

Proposition \ref{construction:first_fan} already gives some information for $\lambda > 1$, but its domain of application ends at $\lambda = 3$.
We need to way to extend this first ``fan'' of lines further up the domain.

Recall that we had obtained the inequality
\begin{align*}
  \mathbb{F}(x, \lambda + x) \geq \mathbb{G}(x,\lambda)
\end{align*}
for all $x \in [0,1]$. We then obtained a lower bound for $\mathbb{G}$ using the concavity of the function
\begin{align*}
  t \mapsto \mathbb{B}(tx, tA, \lambda).
\end{align*}
In particular, this led to $\mathbb{G}(x, \lambda) \geq \frac{1}{2}\mathbb{F}(2x, \lambda)$, but there is an obvious problem when $x > \frac{1}{2}$ as we would be evaluating $\mathbb{F}$ outside its domain.

We can get around this issue by applying concavity along a different path. In particular, let $\frac{1}{2} \leq x \leq 1$, then
\begin{align*}
  \mathbb{G}(x,\lambda) &= \mathbb{B}\Big(x \cdot 1, x \cdot \frac{1}{x}, \lambda\Big) \\
  &\geq x \mathbb{B}\Big(1, \frac{1}{x}, \lambda\Big) \\
  &= x \mathbb{B}\Big(1, \frac{1}{2x} \cdot 2, \lambda\Big) \\
  &\geq \frac{1}{2} \mathbb{B}(1, 2, \lambda).
\end{align*}
That is
\begin{align} \label{construction:T:extended1}
  \mathbb{F}(x,\lambda + x) \geq \mathbb{G}(x,\lambda) \geq \frac{1}{2}\mathbb{F}(1, \lambda) \quad \text{whenever} \quad \frac{1}{2} \leq x \leq 1.
\end{align}
Combining this with \eqref{construction:T} we can write
\begin{align*}
  \mathbb{F}\Big(\frac{x}{2}, \lambda + \frac{x}{2}\Big) \geq \begin{cases}
    \frac{1}{2}\mathbb{F}(x,\lambda) &\text{if }x \in [0,1] \\
    \frac{1}{2}\mathbb{F}(1,\lambda) &\text{if }x \in [1,2]
  \end{cases}
\end{align*}

This case analysis suggests the following extension of inequality \eqref{construction:T} which may aid the geometric interpretation.
Define the extension $\widetilde{\mathbb{F}}(x,\lambda) = \mathbb{F}(\min(1,x), \lambda)$, then
\begin{align*}
  \widetilde{\mathbb{F}}(T(x,\lambda)) \geq \frac{1}{2}\widetilde{\mathbb{F}}(x,\lambda).
\end{align*}

Recall from \eqref{construction:x=1:full} that we already have a lower bound for $\mathbb{F}(1,\lambda)$, so we can use the inequality above
to obtain
\begin{align} \label{construction:interm1}
  \mathbb{F}(x, 2 + x) \geq \frac{1}{2}\mathbb{F}(1, 2) \geq \frac{1}{2}
\end{align}
for all $\frac{1}{2} \leq x \leq 1$.
This extends the lower bound given by Proposition \ref{construction:first_fan} from $F_1$ to the point $(1, 3)$, as depicted in Figure \ref{fig:extension:1}.
\begin{figure}[H]
  \begin{tikzpicture}[scale=0.1]
    \draw[thick, -] (0, 0) -- (64, 0);
    \draw[thick, -] (64, 0) -- (64, 25);
    \draw[thick, -] (0, 0) -- (0, 25);

    \coordinate (O) at (0, 0);

    \coordinate (F0) at (64, 16);
    \coordinate (F1) at (32, 20);
    \coordinate (F2) at (16, 22);
    \coordinate (F3) at (8, 23);
    \coordinate (F4) at (4, 23.5);
    \coordinate (F5) at (2, 23.75);

    \coordinate (F01) at (64, 24);

    \fill (F0) circle (10pt);
    \fill (F1) circle (10pt);
    \fill (F2) circle (10pt);
    \fill (F3) circle (10pt);
    \fill (F4) circle (10pt);
    \fill (F5) circle (10pt);

    \fill (F01) circle (10pt);

    \draw[irinablue, thick, -] (O) -- (F0);
    \draw[irinablue, thick, -] (O) -- (F1);
    \draw[irinablue, thick, -] (O) -- (F2);
    \draw[irinablue, thick, -] (O) -- (F3);
    \draw[irinablue, thick, -] (O) -- (F4);
    \draw[irinablue, thick, -] (O) -- (F5);

    \draw[irinared2, thick, -] (F1) -- (F01);

    \node at ($(F1) + (-2, 2)$) {$F_1$};
    \node at ($(F01) + (5, 0)$) {$(1,3)$};
    \node at ($(F0) + (5, 0)$) {$(1,2)$};
  \end{tikzpicture}
  \caption{Representation of the first extension \eqref{construction:interm1}.}
  \label{fig:extension:1}
\end{figure}
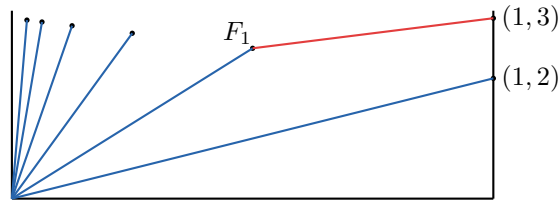

In particular
\begin{align*}
  \mathbb{F} \geq \frac{1}{2} \quad \text{ on the piecewise-linear curve through the points $(0,0) \to F_1 \to (1, 3)$}.
\end{align*}
One can now now use \eqref{construction:T} to extend all the other lines given in Proposition \ref{construction:first_fan}, specifically by iteratively applying $T$ to the segment joining $F_1$ and $(1,3)$.
\begin{figure}[H]
  \begin{tikzpicture}[scale=0.1]
    \draw[thick, -] (0, 0) -- (64, 0);
    \draw[thick, -] (64, 0) -- (64, 25);
    \draw[thick, -] (0, 0) -- (0, 25);

    \coordinate (O) at (0, 0);

    \coordinate (F0) at (64, 16);
    \coordinate (F1) at (32, 20);
    \coordinate (F2) at (16, 22);
    \coordinate (F3) at (8, 23);
    \coordinate (F4) at (4, 23.5);
    \coordinate (F5) at (2, 23.75);

    \coordinate (F01) at (64, 24);
    \coordinate (F11) at (32, 28);
    \coordinate (F21) at (16, 30);
    \coordinate (F31) at (8, 31);
    \coordinate (F41) at (4, 31.5);
    \coordinate (F51) at (2, 31.75);

    \fill (F0) circle (10pt);
    \fill (F1) circle (10pt);
    \fill (F2) circle (10pt);
    \fill (F3) circle (10pt);
    \fill (F4) circle (10pt);
    \fill (F5) circle (10pt);

    \fill (F01) circle (10pt);
    \fill (F11) circle (10pt);
    \fill (F21) circle (10pt);
    \fill (F31) circle (10pt);
    \fill (F41) circle (10pt);

    \draw[irinablue, thick, -] (O) -- (F0);
    \draw[irinablue, thick, -] (O) -- (F1);
    \draw[irinablue, thick, -] (O) -- (F2);
    \draw[irinablue, thick, -] (O) -- (F3);
    \draw[irinablue, thick, -] (O) -- (F4);
    \draw[irinablue, thick, -] (O) -- (F5);

    \draw[irinared2, thick, -] (F1) -- (F01);
    \draw[irinared2, thick, -] (F2) -- (F11);
    \draw[irinared2, thick, -] (F3) -- (F21);
    \draw[irinared2, thick, -] (F4) -- (F31);
    \draw[irinared2, thick, -] (F5) -- (F41);

    \node at ($(F01) + (5, 0)$) {$(1,3)$};
    \node at ($(F0) + (5, 0)$) {$(1,2)$};
  \end{tikzpicture}
  \caption{Iterative application of $T$ leading to the first few extensions.}
  \label{fig:extension:2}
\end{figure}
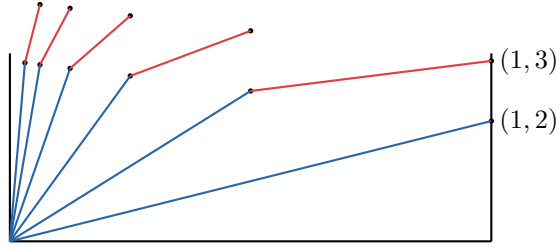

We will use these ideas to produce a ``discrete foliation'' of the $(x,\lambda)$ plane, formed by extending the curves given in Proposition \ref{construction:first_fan}.
In order to state a precise theorem, let us recall the family of points defined in the introduction
\begin{align*}
  F_k^m = \Big( \frac{1}{2^k}, \: m-k+3-\frac{1}{2^{k}} \Big)
\end{align*}
for $0 \leq k \leq m$.
In particular, $F_k = F_k^k$ in the notation of Proposition \ref{construction:first_fan}.

Now define the curve $\Gamma_m$ as the piecewise-linear curve given by
\begin{align*}
  (0,0) \to F_m^m \to \dots \to F_2^m \to F_1^m \to F_0^m.
\end{align*}

\begin{theorem} \label{construction:skeleton}
  Let $\lambda > 0$, then
  \begin{align} \label{construction:full_f}
    \mathbb{F}(x,\lambda) \geq 2^{-m} \quad \text{for every $(x,\lambda)$ on the curve } \Gamma_m.
  \end{align}
\end{theorem}
\begin{proof}
  For every $0 \leq k \leq m$ define the segments
  \begin{align*}
    \ell_k^m = \begin{cases}
      F_k^m \to F_{k+1}^m &\text{if } k < m \\
      F_k^m \to (0,0) &\text{if } k=m.
    \end{cases}
  \end{align*}
  So the curve $\Gamma_m$ consists of the segments $\ell_0^m, \ell_1^m, \dots, \ell_m^m$. We will show
  \begin{align} \label{construction:interm3}
    \mathbb{F}(x,\lambda) \geq 2^{-m-k} \qquad \text{for every $(x,\lambda) \in \ell_k^{m+k}$}
  \end{align}
  for all integers $m,k \geq 0$, which is equivalent to \eqref{construction:full_f}.
  The key idea is that $T$ maps $\ell_k^m$ onto $\ell_{k+1}^{m+1}$, this will allow us to set up an induction argument.

  Proposition \ref{construction:first_nontrivial} provides
  \begin{align*}
    \mathbb{F}(x,\lambda) \geq 1 \qquad \text{for every $(x,\lambda) \in \ell_0^0$}.
  \end{align*}
  We can obtain analogous estimates for $\mathbb{F}$ on $\ell_0^m$ for $m \geq 1$ arguing as above, using inequality \eqref{construction:T:extended1}.
  In particular, from \eqref{construction:T:extended1} and \eqref{construction:x=1:full} we have
  \begin{align*}
    \mathbb{F}(x,1+m + x) &\geq \frac{1}{2}\mathbb{F}(1, 1+m) \geq 2^{-m}
  \end{align*}
  for all intergers $m \geq 1$ and $\frac{1}{2} \leq x \leq 1$.

  Note that the points $\{(x, 1+m+x):\, \frac{1}{2} \leq x \leq 1\}$ form precisely the segment $\ell_0^m$, thus we have shown
  \begin{align} \label{construction:interm2}
    \mathbb{F}(x,\lambda) \geq 2^{-m} \qquad \text{for all $(x,\lambda) \in \ell_0^m$}
  \end{align}
  for every $m \geq 0$. This is our base case.

  As mentioned above, $T$ maps $\ell_k^m$ onto $\ell_{k+1}^{m+1}$, so we can ``jump'' with \eqref{construction:T} lower bounds on $\ell_k^{m+k}$ to lower bounds on $\ell_{k+1}^{m+k+1}$.
  In particular, assume by induction that \eqref{construction:interm3} holds for some $k$ and let $(x,\lambda) \in \ell_{k+1}^{m+k+1}$.
  The preimage $T^{-1}(x,\lambda)$ is contained in $\ell_k^{m+k}$, so by \eqref{construction:T} and the induction hypothesis
  \begin{align*}
    \mathbb{F}(x,\lambda) &\geq \frac{1}{2}\mathbb{F}(T^{-1}(x,\lambda)) \\
    &\geq \frac{1}{2} 2^{-m-k} = 2^{-m-k-1},
  \end{align*}
  which proves \eqref{construction:interm3} for $k+1$ instead of $k$.
\end{proof}
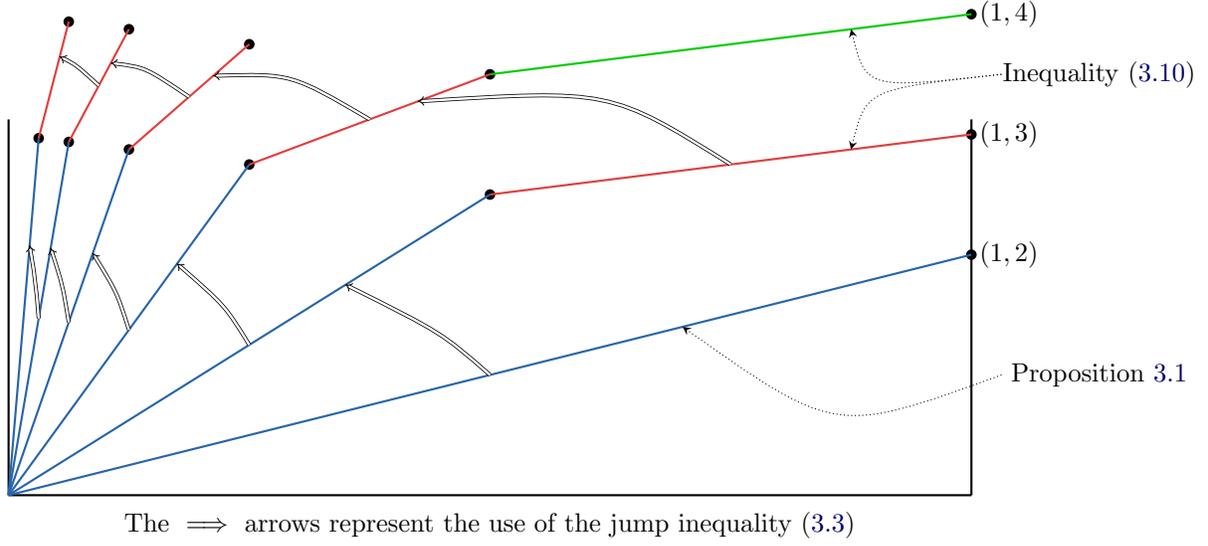
\begin{figure}[H]
  \begin{tikzpicture}[scale=0.2]
    \draw[thick, -] (0, 0) -- (64, 0);
    \draw[thick, -] (64, 0) -- (64, 25);
    \draw[thick, -] (0, 0) -- (0, 25);

    \coordinate (O) at (0, 0);

    \coordinate (F00) at (64, 16);
    \coordinate (F11) at (32, 20);
    \coordinate (F22) at (16, 22);
    \coordinate (F33) at (8, 23);
    \coordinate (F44) at (4, 23.5);
    \coordinate (F55) at (2, 23.75);

    \coordinate (F01) at (64, 24);
    \coordinate (F12) at (32, 28);
    \coordinate (F23) at (16, 30);
    \coordinate (F34) at (8, 31);
    \coordinate (F45) at (4, 31.5);

    \coordinate (F02) at (64, 32);

    \fill (F00) circle (10pt);
    \fill (F11) circle (10pt);
    \fill (F22) circle (10pt);
    \fill (F33) circle (10pt);
    \fill (F44) circle (10pt);
    \fill (F55) circle (10pt);

    \fill (F01) circle (10pt);
    \fill (F12) circle (10pt);
    \fill (F23) circle (10pt);
    \fill (F34) circle (10pt);
    \fill (F45) circle (10pt);

    \fill (F02) circle (10pt);

    \draw[irinablue, thick, -] (O) -- (F00);
    \draw[irinablue, thick, -] (O) -- (F11);
    \draw[irinablue, thick, -] (O) -- (F22);
    \draw[irinablue, thick, -] (O) -- (F33);
    \draw[irinablue, thick, -] (O) -- (F44);
    \draw[irinablue, thick, -] (O) -- (F55);

    \draw[irinared2, thick, -] (F11) -- (F01);
    \draw[irinared2, thick, -] (F22) -- (F12);
    \draw[irinared2, thick, -] (F33) -- (F23);
    \draw[irinared2, thick, -] (F44) -- (F34);
    \draw[irinared2, thick, -] (F55) -- (F45);

    \draw[green1, thick, -] (F12) -- (F02);

    \node at ($(F02) + (2.5, 0)$) {$(1,4)$};
    \node at ($(F01) + (2.5, 0)$) {$(1,3)$};
    \node at ($(F00) + (2.5, 0)$) {$(1,2)$};

    %

    \draw[line width=0.3pt, double distance=0.9pt, arrows = {-Implies[]}] ($0.5*(O)+0.5*(F00)$) .. controls ($0.3*(F00) + 0.3*(F11)$) .. ($0.3*(O) + 0.7*(F11)$);
    \draw[line width=0.3pt, double distance=0.9pt, arrows = {-Implies[]}] ($0.5*(O)+0.5*(F11)$) .. controls ($0.3*(F11) + 0.3*(F22)$) .. ($0.3*(O) + 0.7*(F22)$);
    \draw[line width=0.3pt, double distance=0.9pt, arrows = {-Implies[]}] ($0.5*(O)+0.5*(F22)$) .. controls ($0.3*(F22) + 0.3*(F33)$) .. ($0.3*(O) + 0.7*(F33)$);
    \draw[line width=0.3pt, double distance=0.9pt, arrows = {-Implies[]}] ($0.5*(O)+0.5*(F33)$) .. controls ($0.3*(F33) + 0.3*(F44)$) .. ($0.3*(O) + 0.7*(F44)$);
    \draw[line width=0.3pt, double distance=0.9pt, arrows = {-Implies[]}] ($0.5*(O)+0.5*(F44)$) .. controls ($0.3*(F44) + 0.3*(F55)$) .. ($0.3*(O) + 0.7*(F55)$);

    \draw[line width=0.3pt, double distance=0.9pt, arrows = {-Implies[]}] ($0.5*(F11)+0.5*(F01)$) .. controls ($0.25*(F01) + 0.75*(F12)$) .. ($0.3*(F22) + 0.7*(F12)$);
    \draw[line width=0.3pt, double distance=0.9pt, arrows = {-Implies[]}] ($0.5*(F22)+0.5*(F12)$) .. controls ($0.225*(F12) + 0.725*(F23)$) .. ($0.3*(F33) + 0.7*(F23)$);
    \draw[line width=0.3pt, double distance=0.9pt, arrows = {-Implies[]}] ($0.5*(F33)+0.5*(F23)$) .. controls ($0.2125*(F23) + 0.7125*(F34)$) .. ($0.3*(F44) + 0.7*(F34)$);
    \draw[line width=0.3pt, double distance=0.9pt, arrows = {-Implies[]}] ($0.5*(F44)+0.5*(F34)$) .. controls ($0.20625*(F34) + 0.70625*(F45)$) .. ($0.3*(F55) + 0.7*(F45)$);

    \draw[densely dotted, -stealth] (66,8) .. controls (55, 4) .. ($0.3*(O) + 0.7*(F00)$);
    \node at ($(66,8) + (6.5,0)$) {Proposition \ref{construction:first_nontrivial}};

    \node at ($(66,28) + (6.5,0)$) {Inequality \eqref{construction:interm2}};
    \draw[densely dotted, -stealth] (66,28) .. controls (57, 27) .. ($0.25*(F11) + 0.75*(F01)$);
    \draw[densely dotted, -stealth] (66,28) .. controls (57, 27) .. ($0.25*(F12) + 0.75*(F02)$);

    \node at ($(32,-2) + (0,0)$) {The $\implies$ arrows represent the use of the jump inequality \eqref{construction:T}};
  \end{tikzpicture}
  \caption{Proof of Theorem \ref{construction:skeleton}}
  \label{fig:skeleton_proof}
\end{figure}

This theorem is the ``skeleton'' from which lower bounds for $\mathbb{B}$ follow for $\lambda > 1$.
First, we can use concavity to reduce matters to obtaining bounds for $\mathbb{F}$.
The next proposition, whose proof has been implicitly used in the arguments above, encapsulates this reduction
\begin{proposition}
  \begin{align*}
    \mathbb{B}(x,A,\lambda) \geq \begin{cases}
      \frac{A}{2}\mathbb{F}\big(\frac{2x}{A}, \lambda\big) &\text{if } 2x \leq A,\\
      \frac{A}{2}\mathbb{F}(1, \lambda) &\text{otherwise}.
    \end{cases}
  \end{align*}
\end{proposition}
\begin{proof}
  Assume first that $2x \leq A$, then the function
  \begin{align*}
    t \mapsto \mathbb{B}(x_0 t, 2t, \lambda)
  \end{align*}
  is non-negative and concave in $t \in [0,1]$ for every fixed $x_0 \in [0,1]$, thus
  \begin{align*}
    \mathbb{B}(x_0 t, 2t, \lambda) \geq t \mathbb{B}(x_0, 2, \lambda) = t \mathbb{F}(x_0, \lambda).
  \end{align*}
  Now choose $x_0$ and $t$ so that $(x_0t, 2t) = (x,A)$, i.e.:
  \begin{align*}
    t = \frac{A}{2} \qquad \text{and} \qquad x_0 = \frac{2x}{A}
  \end{align*}
  and the claim follows since $t,x_0 \in [0,1]$.

  When $2x > A$ we can argue as in the proof of \eqref{construction:T:extended1}:
  \begin{align*}
    \mathbb{B}(x,A,\lambda) &= \mathbb{B}\Big(x\cdot 1, x \cdot \frac{A}{x}, \lambda\Big) \\
    &\geq x \mathbb{B}\Big(1, \frac{A}{x}, \lambda) \\
    &= x \mathbb{B}\Big(1, \frac{A}{2x} \cdot 2, \lambda) \\
    &\geq x \frac{A}{2x} \mathbb{B}(1,2\lambda) \\
    &= \frac{A}{2}\mathbb{F}(1,\lambda),
  \end{align*}
  where we have used the concavity of $\mathbb{B}(\cdot, \cdot, \lambda)$ along the paths parametrized by $(t, At/x)$ and $(1, 2t)$
  for $t \in [0,1]$.
\end{proof}

\begin{figure}[H]
  \begin{center}
    \includegraphics[width=8cm]{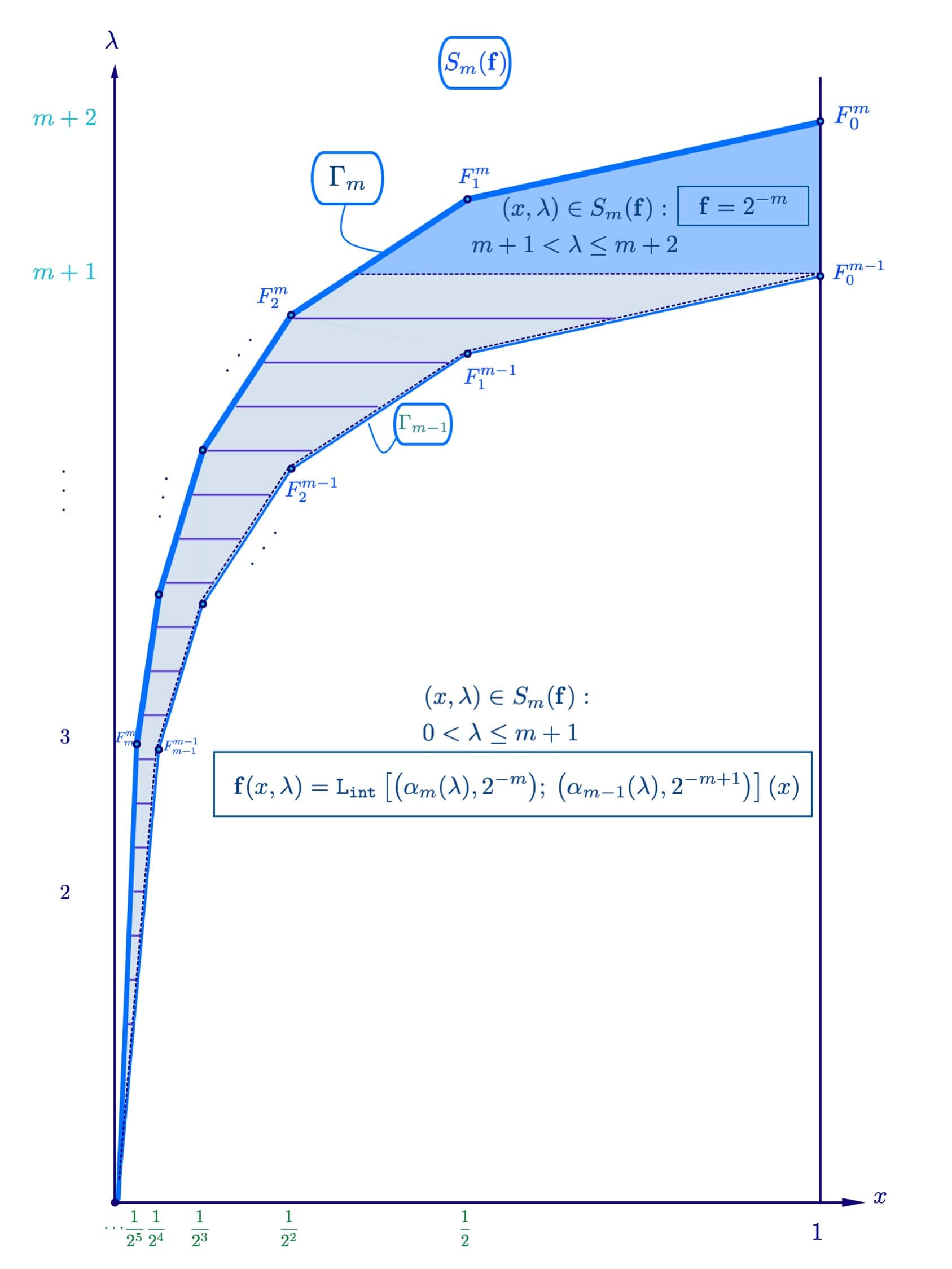}
  \end{center}
  \caption{Sketch of the function $\candidatef$ on $S_m(\candidatef)$.}
  \label{fig:xlplanes-f}
\end{figure}

It remains to fill-in the lower bounds for $\mathbb{F}$ given in Theorem \ref{construction:skeleton}.
The idea is to use the concavity of $\mathbb{F}(\cdot, \lambda)$ in the same way we did when proving \eqref{construction:lerp1}. 
To this end, let $\gamma_m : [0,1] \to \mathbb{R}$ be the function whose graph is $\Gamma_m$, i.e.: the piecewise-linear function
with vertices
\begin{align*}
  \gamma_m(2^{-k}) = m-k+3-\frac{1}{2^k}
\end{align*}
for $0 \leq k \leq m$. Also, let $\alpha_m : [0, m+2] \to [0,1]$ be the inverse function of $\gamma_m$.

Define now the subdomains
\begin{align*}
  S_m(\candidatef) = \begin{cases}
    \{(x,\lambda) \in (0,1] \times \mathbb{R}_{\geq 0}:\, \lambda \leq \gamma_0(x)\} &\text{if }m = 0 \\
    \{(x,\lambda) \in (0,1] \times \mathbb{R}_{\geq 0}:\, \gamma_{m-1}(x) < \lambda \leq \gamma_m(x)\} &\text{if }m \geq 1
  \end{cases}
\end{align*}
Finally, for $\lambda >1$ define $\candidatef(x,\lambda)$ (recall that we had already defined this for $\lambda \leq 1$) as
follows: when $x = 0$ set $\candidatef(x,\lambda) = 0$, set $\candidatef(x,\lambda) = 1$ for all $(x,\lambda) \in S_0$, and
\begin{align} \label{e:f-val-sm}
  \candidatef(x,\lambda) := \begin{cases}
    2^{-m} &\text{if } m + 1< \lambda \leq m+2 \\
    \iLint\big[(\alpha_m(\lambda), 2^{-m}), (\alpha_{m-1}(\lambda), 2^{-m+1})\big](x) &\text{otherwise},
  \end{cases}
\end{align}
when $(x,\lambda) \in S_m(\candidatef)$ with $m \geq 1$ (See Figure \ref{fig:xlplanes-f}).

\begin{theorem} \label{construction:lambda_gt_1}
  \begin{align*}
    \mathbb{F}(x,\lambda) \geq \candidatef(x,\lambda)
  \end{align*}
  for all $0 \leq x \leq 1$ and $\lambda > 0$.
\end{theorem}
\begin{proof}
  Let $(x,\lambda) \in S_m(\candidatef)$. When $m = 0$ then one can write $(x,\lambda)$ as a convex combination of a point $(x_0,\lambda)$
  on $\Gamma_0$ and $(1,\lambda)$, on both of these points $\mathbb{F} \geq 1$ so we can conclude that $\mathbb{F}(x,\lambda) \geq \candidatef(x,\lambda)$.

  Suppose now that $m \geq 1$ and that $\lambda \leq m+1$. We can write $(x,\lambda)$ as a convex combination of a
  point $(\alpha_m(\lambda), \lambda) \in \Gamma_m$ and a point $(\alpha_{m-1}(\lambda), \lambda) \in \Gamma_{m-1}$. The values of
  $\mathbb{F}$ at these points are at least $2^{-m}$ and $2^{-m+1}$ respectively by Theorem \ref{construction:skeleton}, so
  \begin{align*}
    \mathbb{F}(x,\lambda) \geq \iLint\big[(\alpha_m(\lambda), 2^{-m}), (\alpha_{m-1}(\lambda), 2^{-m+1})\big](x) = \candidatef(x,\lambda).
  \end{align*}

  Finally, if $m < \lambda \leq m+1$ then we can write $(x,\lambda)$ as the convex combination of a point
  $(\alpha_m(\lambda), \lambda)$ and a point $(1, \lambda)$. By Theorem \ref{construction:skeleton} and \eqref{construction:x=1:full}
  the value of $\mathbb{F}$ at both of these points is at least $2^{-m}$, so by concavity we conclude $\mathbb{F}(x,\lambda) \geq 1 = \candidatef(x,\lambda)$
  and we are done.
\end{proof}

We can now complete the definition of $\candidateB$ for $\lambda > 1$:
\begin{align*}
  \candidateB(x,A,\lambda) = \begin{cases}
    \frac{A}{2}\candidatef\big(\frac{2x}{A}, \lambda\big) &\text{if } 2x \leq A,\\
    \frac{A}{2}\candidatef(1,\lambda) &\text{otherwise}.
  \end{cases}
\end{align*}

The combination of Theorem \ref{construction:lambda_leq_1} and Theorem \ref{construction:lambda_gt_1} yields
\begin{corollary}
  \begin{align*}
    \mathbb{B}(x,A,\lambda) \geq \candidateB(x,A,\lambda)
  \end{align*}
  for all $(x,A,\lambda) \in \Omega_{\mathbb{B}}$.
\end{corollary}


\section{Verifying the Main Inequality}
\label{S:4}

In this section we show that $\bfb\in\mcb$.
Since $\bfb$ satisfies the Obstacle Condition by construction, we are left to verify that $\candidateB$ satisfies the jump inequality \eqref{bellman:jump} and that $\bfb(\cdot,\cdot,\lambda)$ is concave.
We will introduce the function $\candidateg$, which is defined analogously to how $\mathbb{G}$ was defined in terms of $\mathbb{F}$:
\begin{align*}
  \candidateg(x,\lambda) := \candidateB(x, 1, \lambda).
\end{align*}

\subsection{The restricted jump inequality}

In this subsection we prove the special case of \eqref{bellman:jump} when $A=1$. Later in the section we will reduce showing that $\candidateB$
satisfies the full jump inequality to this restricted case.
\begin{proposition}
  \label{P:reduction}
 The function $\candidateB(\cdot,\lambda)$ satisfies the jump inequality:
  \begin{equation}
    \label{e:Ja2}
    \candidatef(x,\lambda+x)\geq \candidateg(x,\lambda).
  \end{equation}
\end{proposition}

We follow a geometric strategy: roughly speaking, we will partition the domain of $\bfg$ into certain subsets of angle sectors. The specific structure of $\bfg$ and $\bff$ will ensure that the jump inequality is satisfied between angle sectors in the $\bfg$-domain and their images in the $\bff$-domain. We begin with some definitions and geometrical lemmas in the $(x,\lambda)$-plane.

\begin{figure}[H]
\centering
\begin{subfigure}{.24\textwidth}
  \centering
  \includegraphics[width=\linewidth]{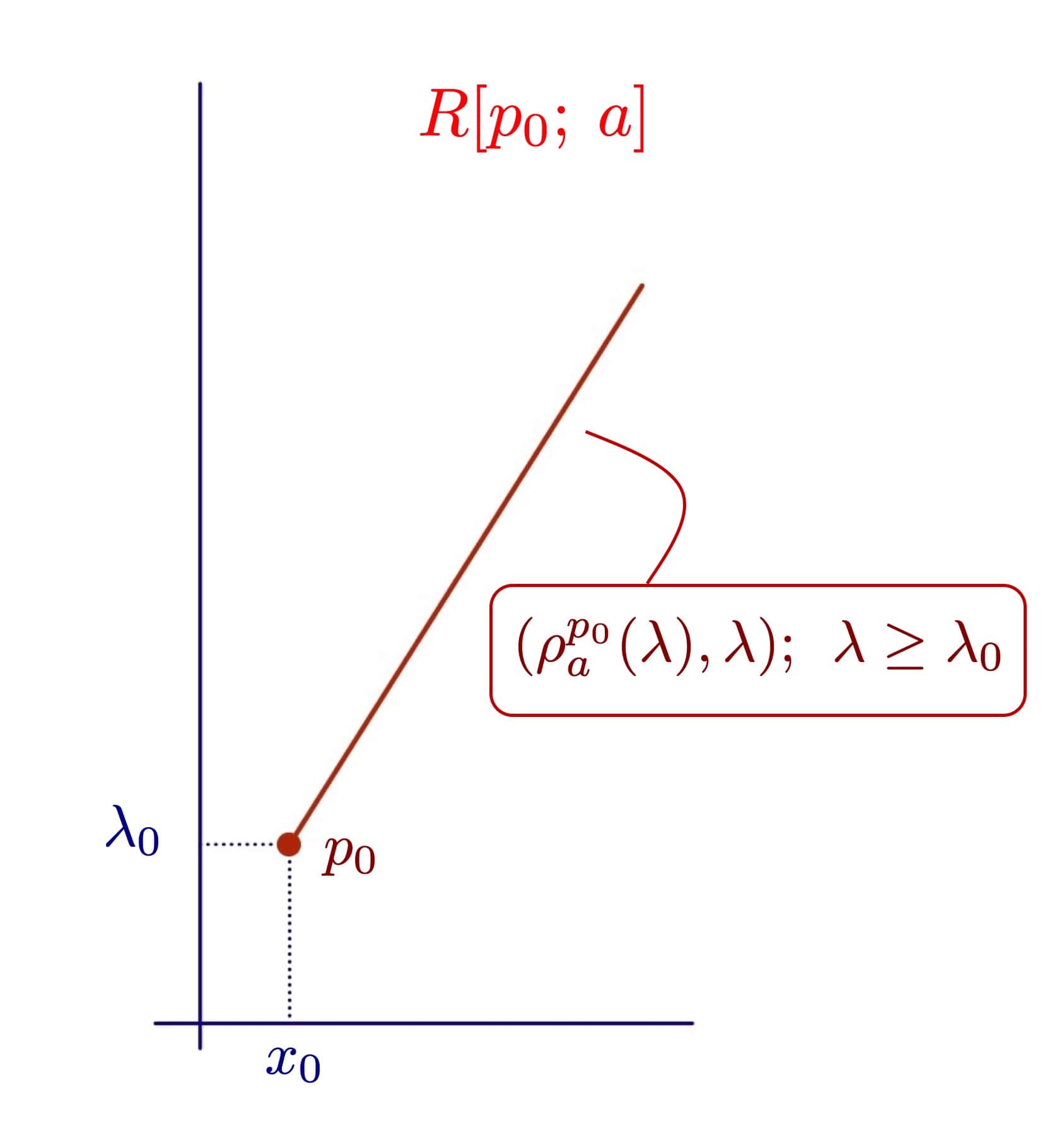}
 \caption{Ray.}
  \label{fig:Jpic1}
\end{subfigure}%
\begin{subfigure}{.24\textwidth}
  \centering
  \includegraphics[width=\linewidth]{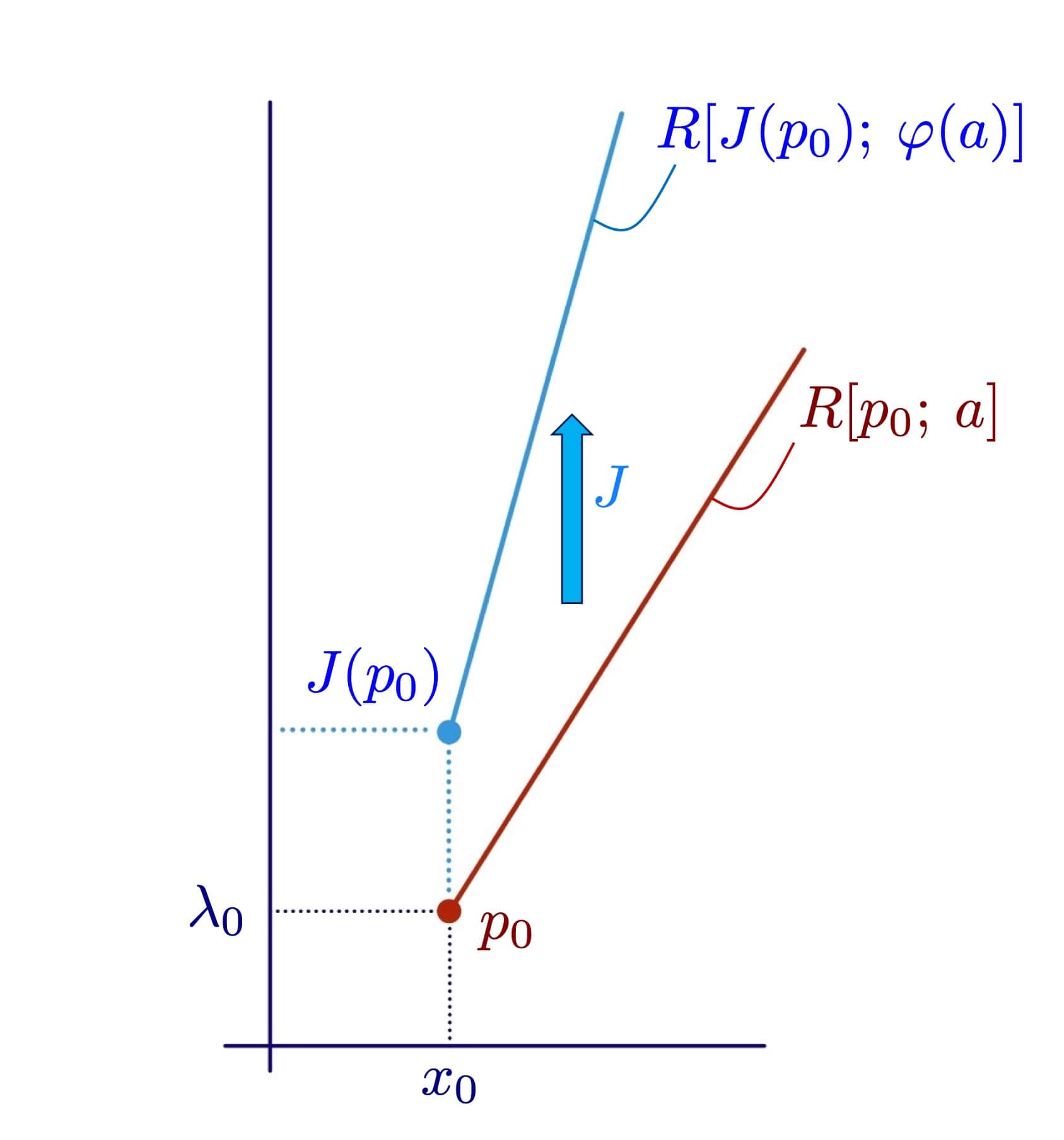}
  \caption{Jump of a ray.}
  \label{fig:Jpic3}
\end{subfigure}
\begin{subfigure}{.24\textwidth}
  \centering
  \includegraphics[width=\linewidth]{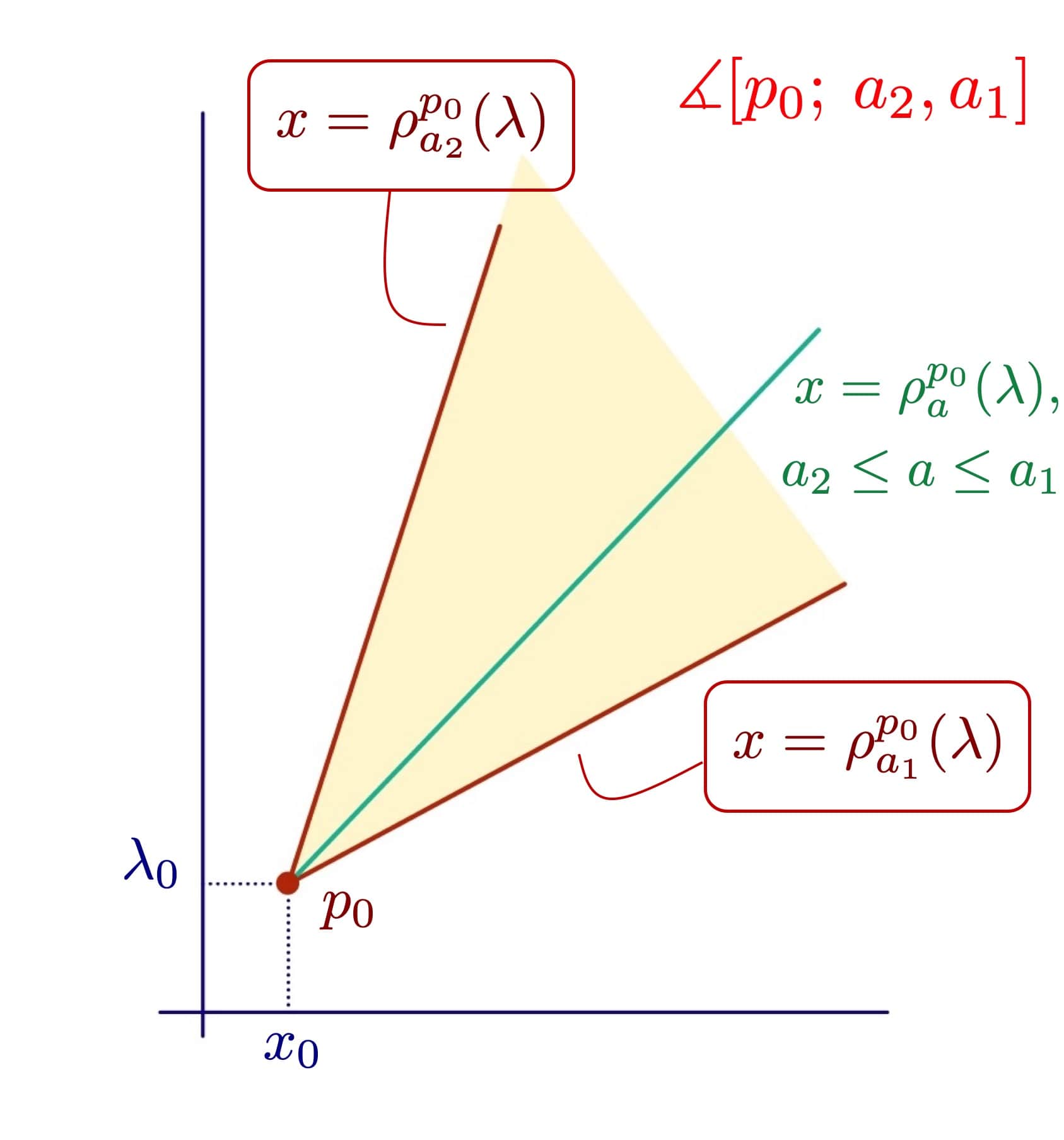}
 \caption{Angle sector.}
  \label{fig:Jpic2}
\end{subfigure}%
\begin{subfigure}{.24\textwidth}
  \centering
  \includegraphics[width=\linewidth]{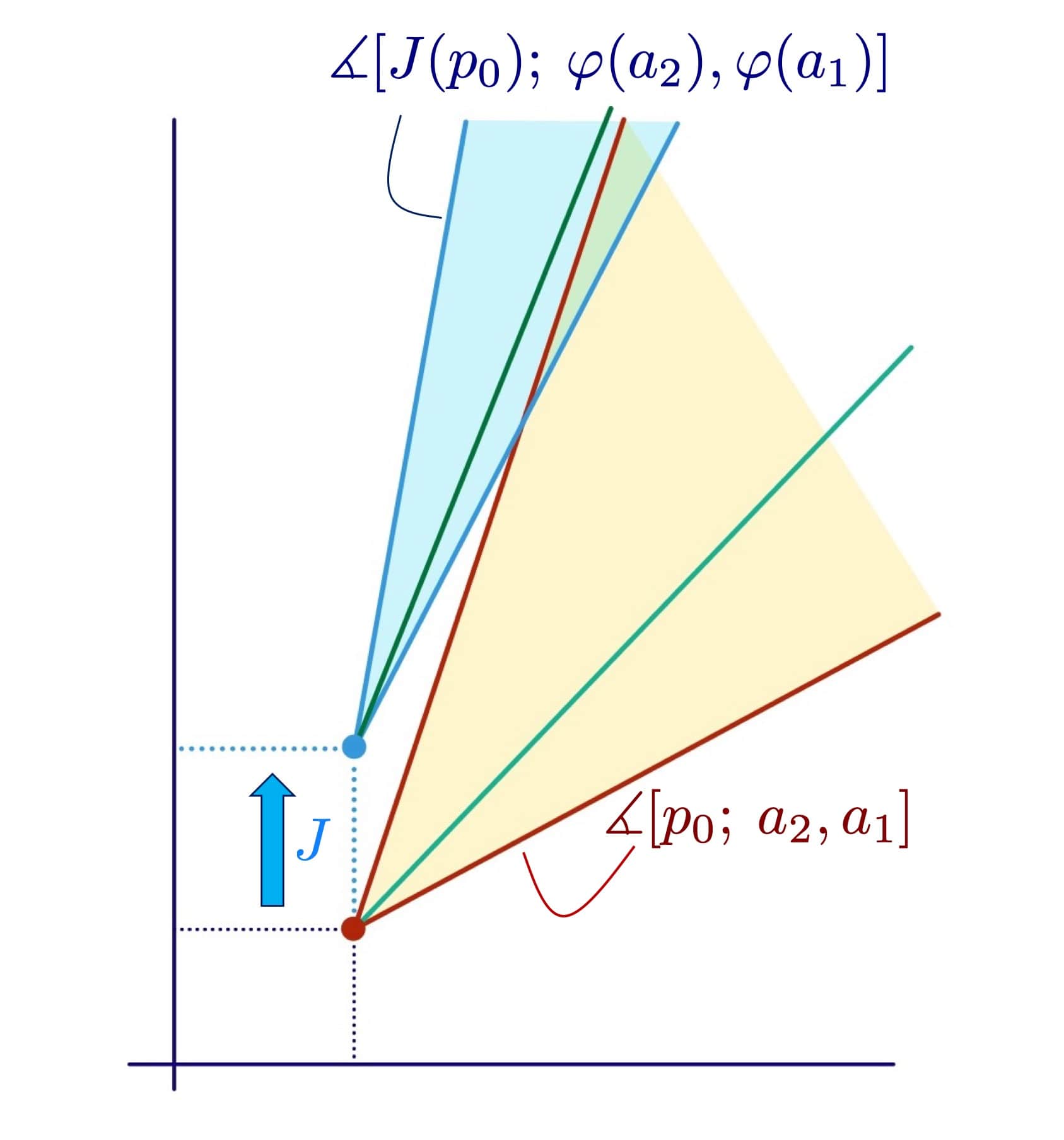}
 \caption{Jump of angle sector.}
  \label{fig:Jpic4}
\end{subfigure}%
\caption{}
\label{fig:Jpics}
\end{figure}

\begin{definition}
Let the \textbf{ray centered at a point $p_0 = (x_0, \lambda_0)$ with parameter $a>0$} (see Figure \ref{fig:Jpic1}) be the set
	$$
	R[p_0; a] := \{(x,\lambda): x=\rho_a^{p_0}(\lambda); \:\: \lambda \geq \lambda_0\},
	$$
where
	$$
	\rho_a^{p_0}(\lambda) := a(\lambda-\lambda_0)+x_0; \:\: \lambda \geq \lambda_0.
	$$
\end{definition}
\noindent Note that $a$ is just the reciprocal of the slope (seen as a function of $x$).

\begin{definition}
Let the \textbf{angle sector centered at $p_0$ with parameters $0<a_2<a_1$} (see Figure \ref{fig:Jpic2}) be the ``fan'' of rays $R[p_0; a]$
centered at $p_0$ with parameter $a_2\leq a\leq a_1$:
\begin{align*}
  \measuredangle[p_0; \: a_2, a_1] := \{ (x, \lambda): \rho_{a_2}^{p_0}(\lambda) \leq x \leq \rho_{a_1}^{p_0}(\lambda); \: \lambda \geq \lambda_0 \}.
\end{align*}
\end{definition}

\begin{lemma}
\label{L:JumpedRays}
For a point $p_0$ in the plane and a parameter $a>0$:
	$$
	J \big( R[p_0; a] \big) = R\big[J(p_0); \varphi(a)\big],
	$$
where
	$$
	\varphi(a) := \frac{a}{1+a}.
  $$
  In other words, $J$ takes a ray centered at $p_0$ with parameter $a>0$ and maps it to the ray centered at $J(p_0)$ with parameter $\varphi(a)$.
\end{lemma}

\begin{proof}
The image of a line under the affine transformation $J$ is the line with the same $\lambda$-intercept but slope increased by $1$.
  By direct calculation,
  \begin{equation*}
    \begin{split}
      J\big[(\rho_a^{p_0}(\lambda), \lambda):\: \lambda\geq \lambda_0\big] & = \big[(\rho_a^{p_0}(\lambda), \lambda + \rho_a^{p_0}(\lambda)): \: \lambda\geq \lambda_0 \big] \\
      & = \big[ \big(\varphi(a)(\lambda-(x_0+\lambda_0))+x_0, \: \lambda\big) :\: \lambda\geq x_0+\lambda_0\big] \\
      & = \big[\big(\rho_{\varphi(a)}^{J(p_0)}(\lambda), \lambda\big):\: y\geq x_0+y_0\big].
    \end{split}
  \end{equation*}
\end{proof}
From this we immediately we have the corresponding result for angle sectors.
This is the key tool used to prove \eqref{e:Ja2}.
\begin{corollary}
  \label{C:JumpedAngles} For a point $p_0$ and parameters $0<a_2<a_1$:
  $$
  J\big(\measuredangle[p_0; a_2, a_1]\big) = \measuredangle \big[J(p_0); \varphi(a_2), \varphi(a_1)\big].
  $$
\end{corollary}

\begin{lemma} \label{L:Lint1}
Let a point $p_0$, parameters $0<a_2<a_1$, and real numbers $v_1, v_2$. Define the function 
	$$
	L : \measuredangle[p_0; a_2, a_1] \rightarrow \mbr,
	$$  
 by
	$$
	L(x, \lambda) := \iLint\big[\big(\rho_{a_2}^{p_0}(\lambda), v_2\big); \: \big(\rho_{a_1}^{p_0}(\lambda), v_1\big)\big](x),
	$$
for all $\lambda\geq \lambda_0$, $\rho_{a_2}^{p_0}(\lambda) \leq x \leq \rho_{a_1}^{p_0}(\lambda)$. Then $L$ is constant along rays centered at $p_0$, and its value along such a ray is 
	\begin{equation} \label{e:JLintRays}
	L\big(\rho_a^{p_0}(\lambda), \lambda\big) = (v_1-v_2) \frac{a-a_2}{a_1-a_2} +v_2, \: \text{ for all } a_2\leq a\leq a_1, \: \lambda\geq \lambda_0.
	\end{equation}
\end{lemma}

\begin{proof}
  Let $(x, \lambda)$ be a point on the ray $R[p_0; a]$, for some $a_2\leq a \leq a_1$. Then
  $$
  L(x,\lambda) = \frac{v_1-v_2}{\rho_{a_1}^{p_0}(\lambda) - \rho_{a_2}^{p_0}(\lambda)} (\rho_a^{p_0}(\lambda) - \rho_{a_2}^{p_0}(\lambda)) +v_2,
  $$
  which simplifies to precisely \eqref{e:JLintRays}
\end{proof}

\begin{corollary} \label{C:JumpLint}
  Let $p_0$ be a point in the plane, parameters $0<a_2<a_1$, and positive real numbers $0<v_2<v_1$. Define the functions $g$ and $f$ as:
  $$
  g:\measuredangle[p_0; a_2, a_1] \rightarrow \mbr; \:\:
  g(x, \lambda) := \iLint\big[(\rho_{a_2}^{p_0}, v_2); \: (\rho_{a_1}^{p_0}, v_1)\big].
  $$
  $$
  f: J\big(\measuredangle[p_0; a_2, a_1]\big) \rightarrow \mbr; \:\:
  f(x, \lambda) := \iLint\big[(\rho_{\varphi(a_2)}^{J(p_0)}, v_2); \: (\rho_{\varphi(a_1)}^{J(p_0)}, v_1)\big].
  $$
  Then
  $$
  f\big(J(x, \lambda)\big) \geq g(x,\lambda),
  $$
  for all $(x, \lambda) \in \measuredangle[p_0; a_2, a_1]$.
\end{corollary}

\begin{proof}
From Lemma \ref{L:Lint1}, $g$ is constant along rays $R[p_0; a]$, where $a_2\leq a\leq a_1$, its value there being
	$$
	(v_1-v_2)\frac{a-a_2}{a_1-a_2} +v_2.
	$$
Under $J$, these rays are mapped to the rays $R[J(p_0); \varphi(a)]$ in the domain of $f$, along which -- again by Lemma \ref{L:Lint1} -- $f$ is constant, with value
	$$
	(v_1-v_2)\frac{\varphi(a) - \varphi(a_2)}{\varphi(a_1)-\varphi(a_2)} +v_2. 
	$$
So, proving $f(J(x,\lambda)) \geq g(x, \lambda)$ becomes equivalent to proving
	$$
	\frac{\varphi(a_1)-\varphi(a_2)}{a_1-a_2} \leq \frac{\varphi(a) - \varphi(a_2)}{a-a_2}, \text{ for all } a_2 <a\leq a_1,
	$$
which is true of any concave function.
\end{proof}


\vspace{0.2in}
In order to prove \eqref{e:Ja2}, it will be convenient to have an explicit expression for $\candidateg$ which exhibits its geometric structure.
When $\lambda \leq 1$ we can give an expression for $\candidateg$ similar to \eqref{construction:flerp}
\begin{align*}
  \candidateg(x,\lambda) = \begin{cases}
    \frac{1}{2}\candidatef(2x,\lambda) &\text{if } x \leq \frac{\lambda}{4} \\
    \iLint\big[ \big(\frac{\lambda}{4}, \frac{1}{2}), (\lambda, 1) \big](x) &\text{if } \frac{\lambda}{4} < x \leq \lambda \\
    1 &\text{if } x > \lambda
  \end{cases}
\end{align*}

To give an explicit expression for $\candidateg(x,\lambda)$ for larger $\lambda$ we can define subdomains $S_m(\candidateg)$ analogous to those
used to define $\candidatef$ in the previous section.

Define the points
\begin{align*}
  G_k^m = \Big(\frac{1}{2^k}, \: m-k+3-\frac{1}{2^{k-1}}\Big)
\end{align*}
for $0 \leq k \leq m$. These are related to the points $F_k^m$ by
\begin{align*}
  G_k^m = C(F^{m-1}_{k-1}) \quad\text{and}\quad F_k^m = J(G^m_k)
\end{align*}
The curve $\breve{\Gamma}_m$, analogous to $\Gamma_m$ (see Figure \ref{fig:xlplanes-g}), is the piecewise-linear curve given by
\begin{align*}
  (0,0) \to G_m^m \to \dots \to G_2^m \to G_1^m \to G_0^m.
\end{align*}
Let $\breve{\gamma}_m$ be the function on $[0,1]$ whose graph is $\breve{\Gamma}_m$, and let $\breve{\alpha}_m : [0, m+1] \to [0,\frac{1}{2}]$
be its inverse function.

The function $\candidateg$ can be written in terms of the domains: for $m \geq 1$
\begin{align*}
  S_m(\candidateg) = \{(x,\lambda) \in (0,1] \times \mathbb{R}_{\geq 0}:\, \breve{\gamma}_{m-1}(x) < \lambda \leq \breve{\gamma}_m(x) \},
\end{align*} and
\begin{align*}
  S_0(\candidateg) = \{(x,\lambda) \in (0,1] \times \mathbb{R}_{\geq 0}:\, \lambda \leq \breve{\gamma}_0(x) \}.
\end{align*}
Observe that $S_m(\candidatef) = J(S_m(\candidateg))$.

\begin{figure}[H]
  \centering
  \begin{subfigure}{0.4875\textwidth}
    \centering
    \includegraphics[width=\linewidth]{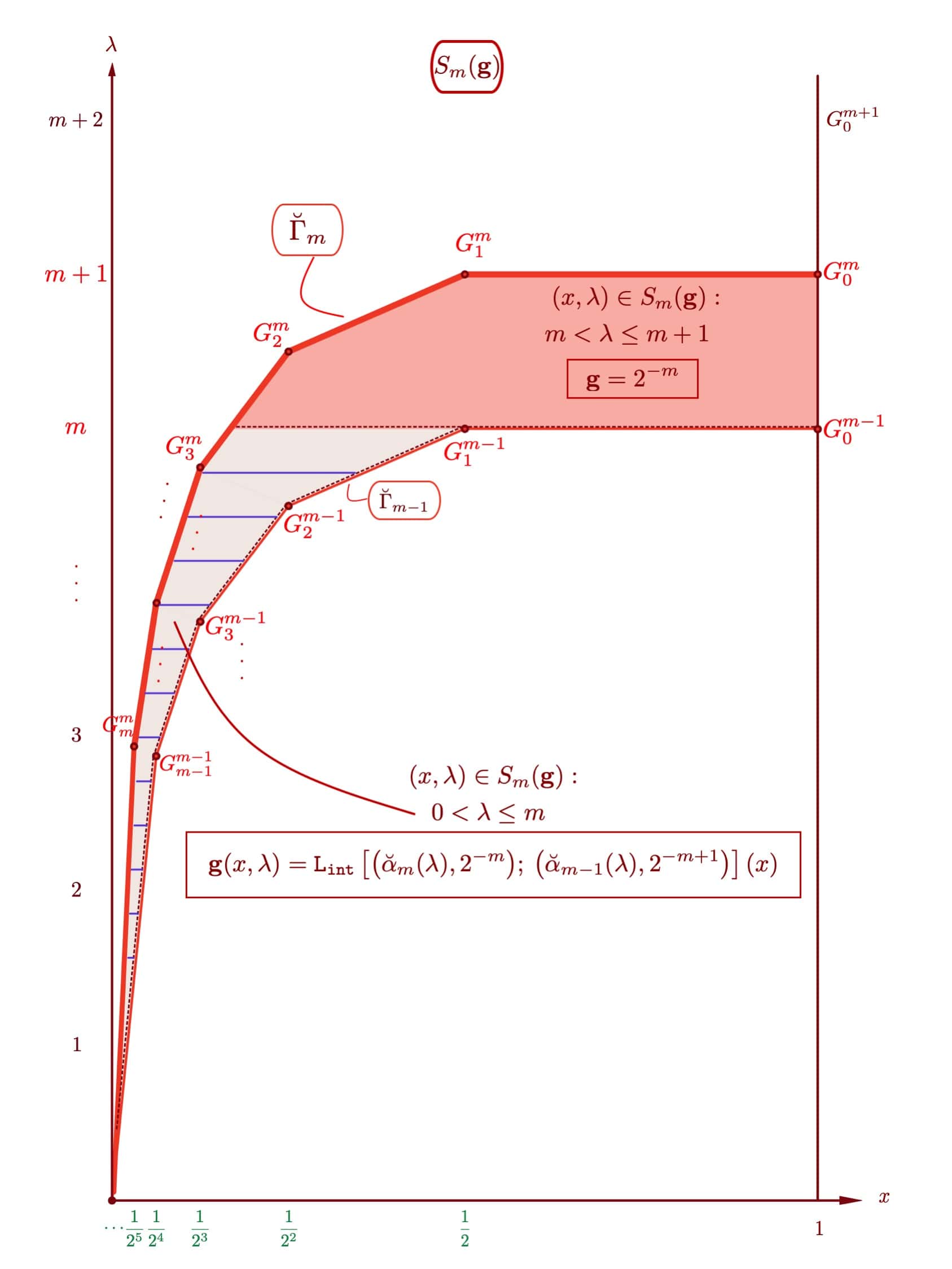}
    \caption{Sketch of the function $\candidateg$ on $S_m(\candidateg)$.}
    \label{fig:xlplanes-g-min}
  \end{subfigure}
  \begin{subfigure}{.4875\textwidth}
    \centering
    \includegraphics[width=\linewidth]{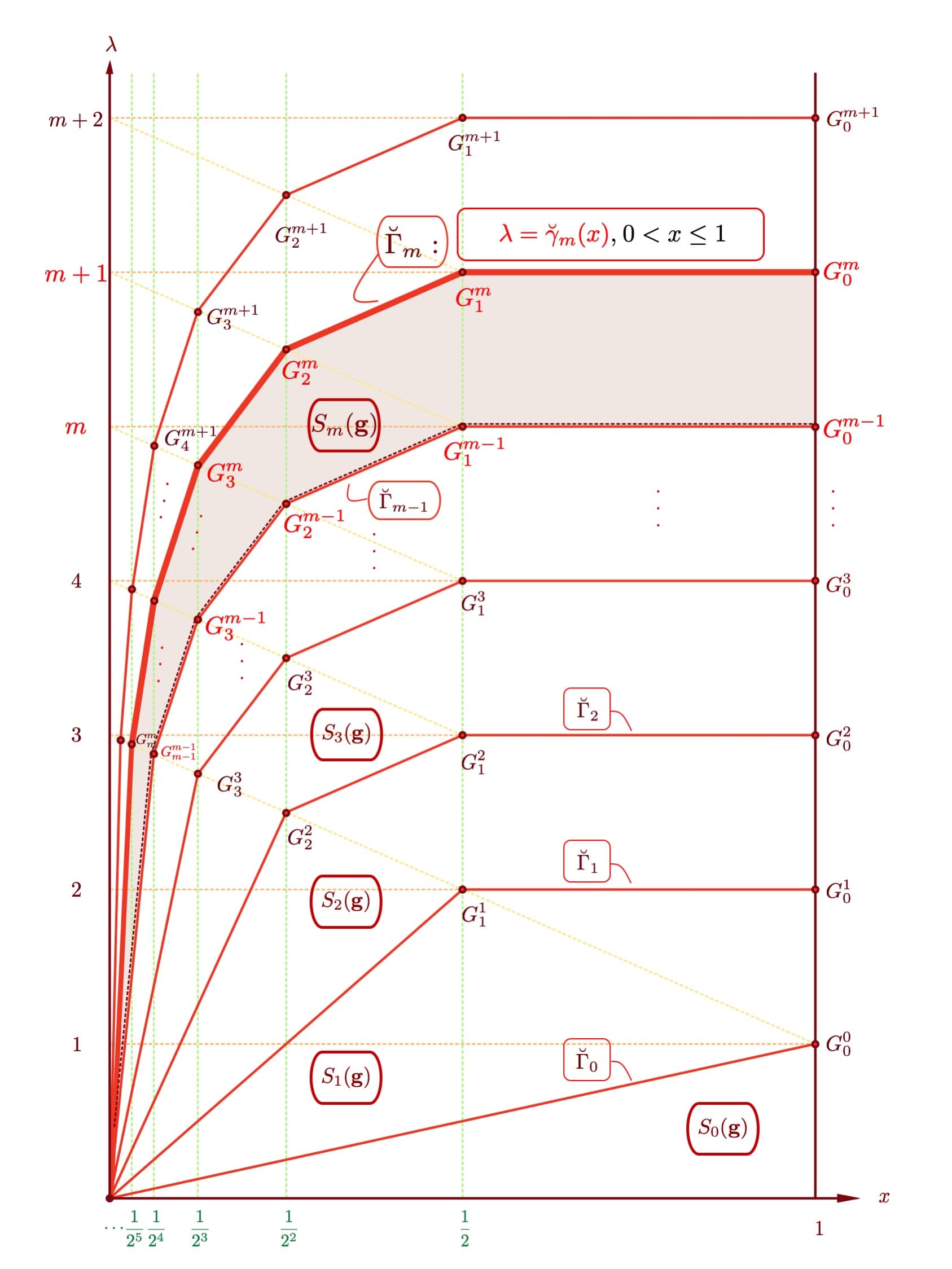}
    \caption{Sketch of the $S_m(\candidateg)$ subdomains.}
    \label{fig:xlplanes-g}
  \end{subfigure}
\caption{}
\end{figure}

With these definitions we have
\begin{align*}
  \bfg(x,\lambda) = \begin{cases}
    2^{-m}, &\text{ if } m<\lambda\leq m+1\\
    \iLint\left[ \left(\breve{\alpha}_m(\lambda), 2^{-m}\right); \: \left(\breve{\alpha}_{m-1}(\lambda), 2^{-m+1}\right)\right](x), &\text{ otherwise.}
  \end{cases}
\end{align*}
for every $(x,\lambda) \in S_m(\candidateg)$ and $m \geq 1$ (see Figure \ref{fig:xlplanes-g-min}), and $\candidateg(x,\lambda)=1$ when $(x,\lambda) \in S_0(\candidateg)$.

\begin{proposition}
  The functions $\bfg$ and $\bff$ satisfy:
  \begin{equation} \label{e:J1}
    \bff\big(J(x, \lambda)\big) \geq \bfg(x, \lambda),
  \end{equation}
  for all $(x,\lambda)$.
\end{proposition}

\begin{proof}
We show that \eqref{e:J1} holds in each strip $S_m(\bfg)$, $m\geq 0$:
	\begin{equation} \label{e:J2}
	\text{If } p \in S_m(\bfg), \text{ then } \bff(J(p)) \geq \bfg(p).
	\end{equation}

\begin{figure}[h]
\centering
\begin{subfigure}{.5\textwidth}
  \centering
  \includegraphics[width=\linewidth]{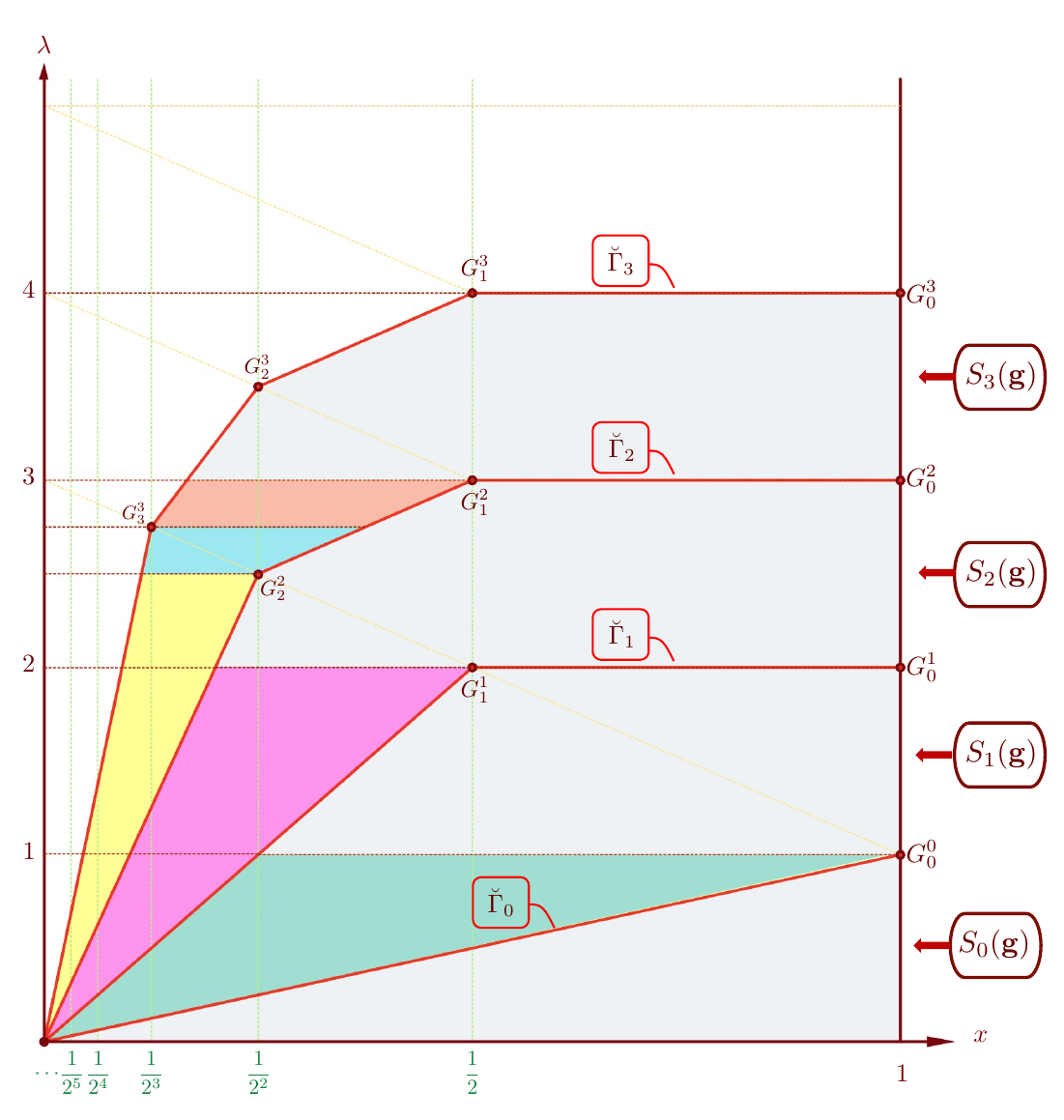}
  \caption{Sketch of $\bfg$: highlighted regions jump to the corresponding regions highlighted in the sketch of $\bff$.}
  \label{fig:SStackg}
\end{subfigure}%
\begin{subfigure}{.5\textwidth}
  \centering
  \includegraphics[width=0.98\linewidth]{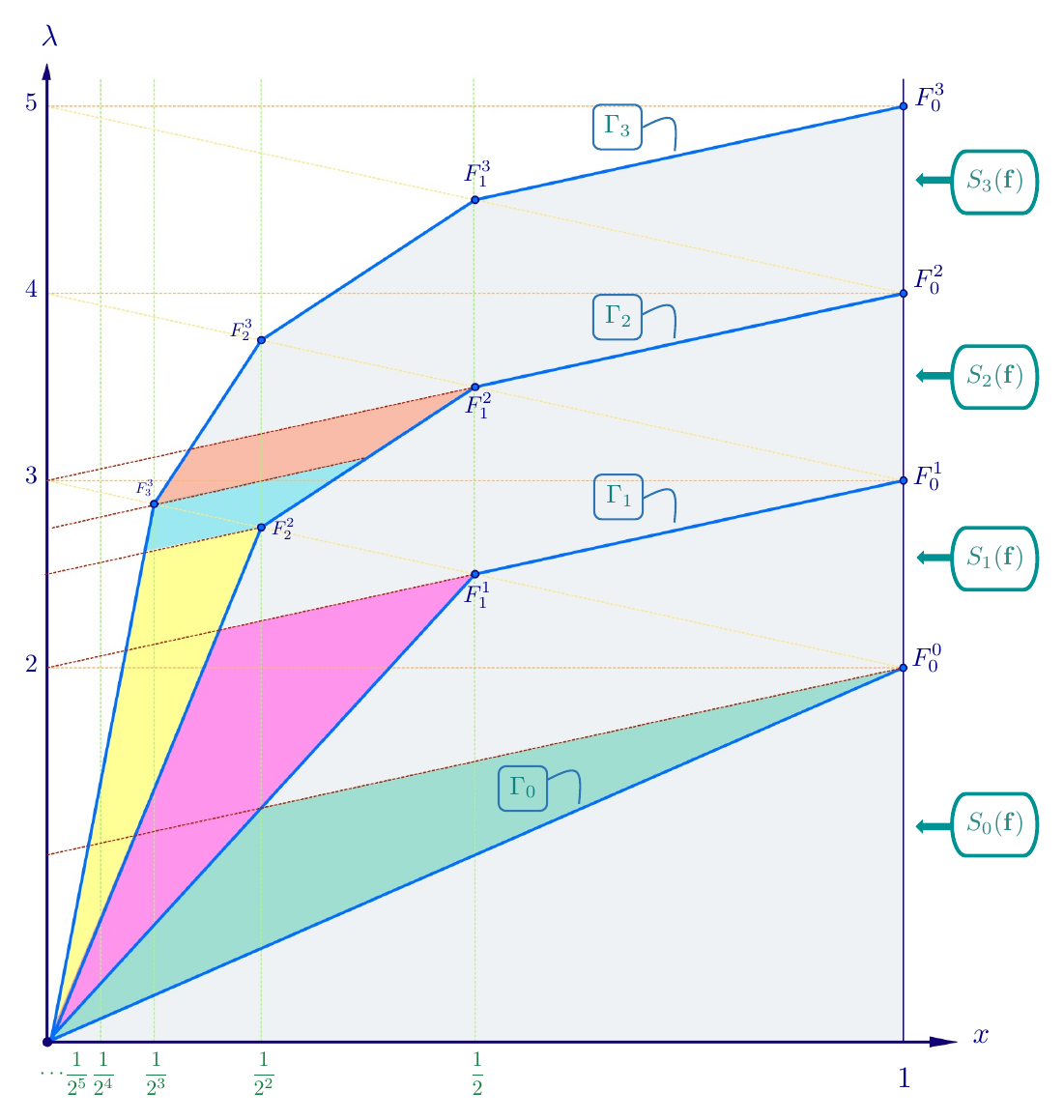}
  \caption{Sketch of $\bff$, showing the images under Jump of corresponding $\bfg$-regions.}
  \label{fig:SStackf}
\end{subfigure}
\caption{}
\label{fig:SStack}
\end{figure}

\begin{figure}
\centering
\begin{subfigure}{.45\textwidth}
  \centering
  \includegraphics[width=\linewidth]{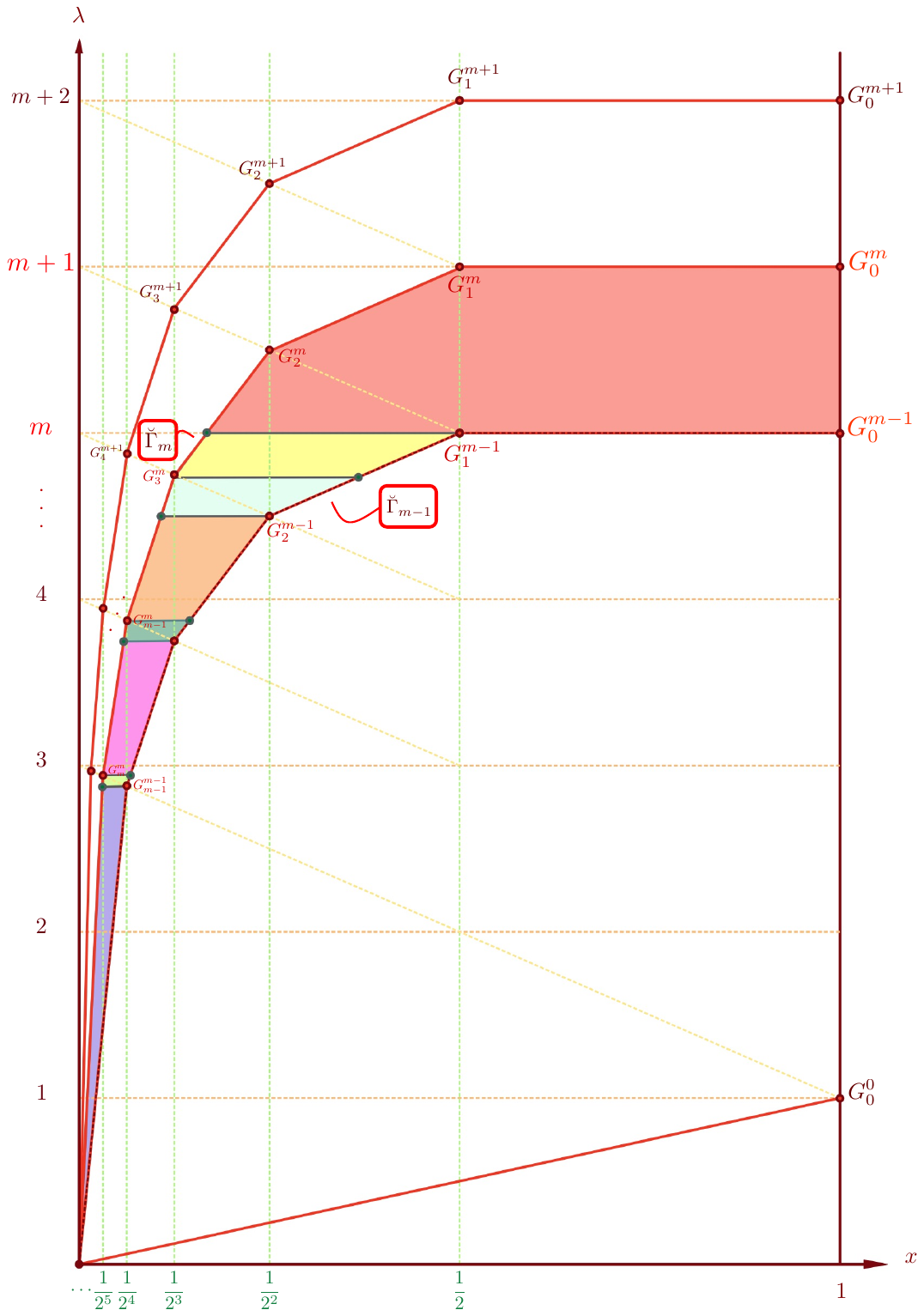}
  \caption{Domain of $\bfg$.}
  \label{fig:Jstripg}
\end{subfigure}%
\begin{subfigure}{.45\textwidth}
  \centering
  \includegraphics[width=\linewidth]{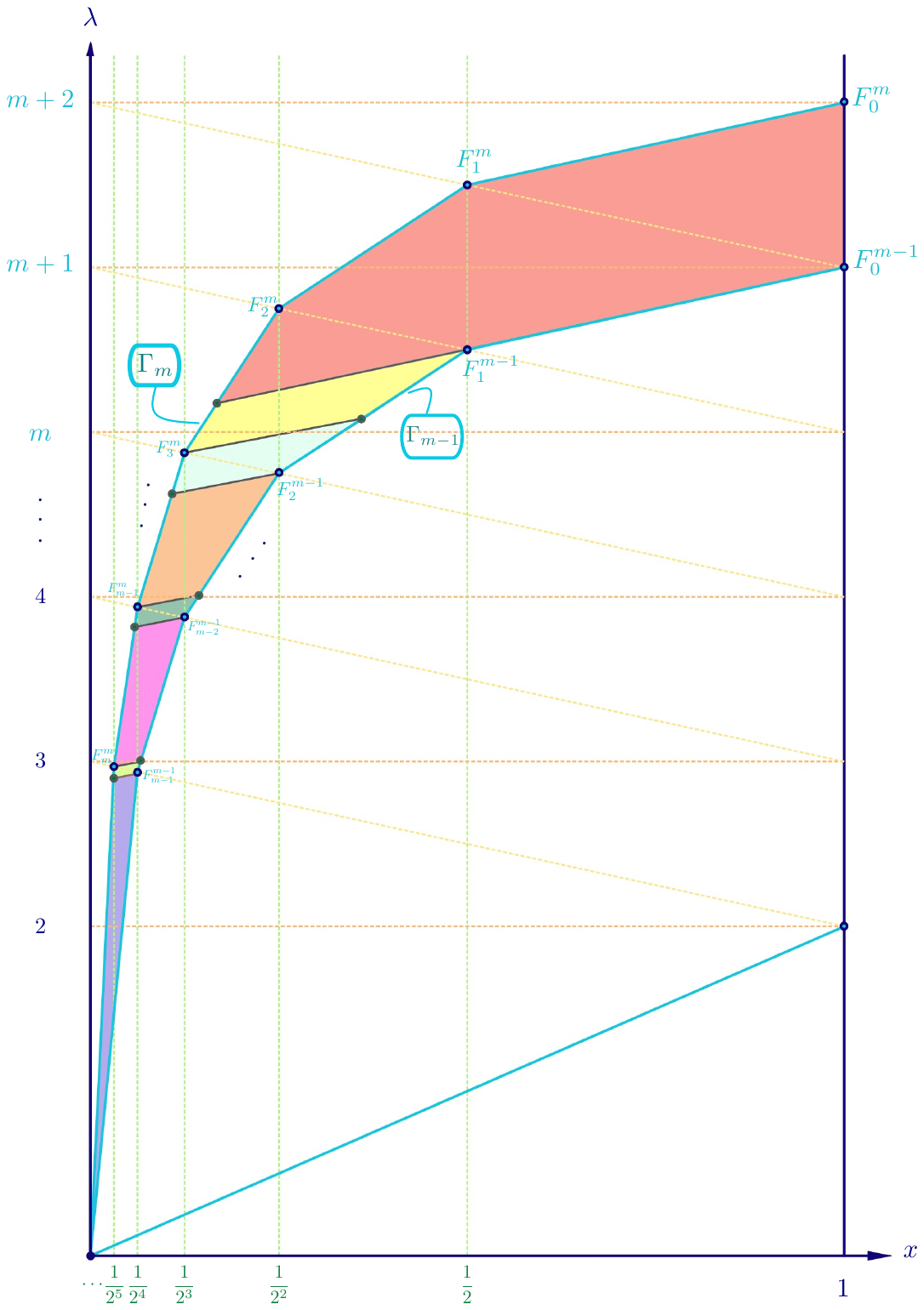}
  \caption{Domain of $\bff$.}
  \label{fig:Jstripf}
\end{subfigure}
\caption{Jump of a strip}
\label{fig:Jstrip}
\end{figure}

  We work our way up, starting with $S_0(\bfg)$ (see Figure \ref{fig:SStack}): here \eqref{e:J1} trivially holds, and is in fact equality -- since both $\bff$ and $\bfg$ are identically $1$ in $S_0(\bff)$ and $S_0(\bfg)$, respectively.
  Now, let us look at $S_1(\bfg)$. First note that \eqref{e:J1} will trivially hold in the trapezoidal region $\{(x,\lambda) \in S_1(\bfg):\, 1<\lambda\leq 2\}$, where $\bfg$ is identically
  $\tfrac{1}{2}$. What remains of $S_1(\bfg)$ is the part contained in the angle sector $\measuredangle[0; \: \tfrac{1}{4}, 1]$, which jumps by construction into exactly the part of $S_1(\bff)$ contained in the angle sector
  $\measuredangle[0; \: \tfrac{1}{5}, \tfrac{1}{2}]$. Then it follows from Corollary \ref{C:JumpLint} that \eqref{e:J1} holds in all of $S_1(\bfg)$.

  The situation is similar for $S_2(\bfg)$. In fact, we can see right away that for every strip $S_m(\bfg)$, $m\geq 1$, \eqref{e:J1} trivially holds in the trapezoidal region
  $\{(x,\lambda) \in S_m(\bfg):\, m<\lambda\leq m+1\}$, where $\bfg$ is identically $2^{-m}$. Similarly, we can quickly check that \eqref{e:J1} will hold for the lowest angular region anchoring $S_m(\bfg)$ to the origin.
  
  The line segments anchoring $\breve{\Gamma}_m$ and $\Gamma_m$ to the origin have equations
  $$
  (0 \to G_m^m]: \: \big[x=y_m\lambda:\: 0<\lambda\leq\tfrac{1}{2^my_m} \big] \:\:\:\text{ and }\:\:\:
  (0 \to F_m^m]: \: \big[x=x_m\lambda:\: 0<\lambda\leq\tfrac{1}{2^mx_m} \big],
  $$
  respectively, where
  $$
  y_m = \frac{1}{3\cdot 2^m - 2} \:\:\:\text{ and }\:\:\: x_m = \frac{1}{3\cdot 2^m -1},
  $$
  for all $m\geq 0$. Note that these sequences satisfy
  $$
  x_m=\varphi(y_m) \:\:\:\text{ and }\:\:\: y_{m+1}=\tfrac{1}{2}x_m, \text{ for all } m\geq 0.
  $$
  For every $S_m(\bfg)$ with $m\geq1$ (see Figure \ref{fig:Jstrip}), the angular region with vertices $0$, $G_{m-1}^{m-1}$ and top side parallel to the $x$-axis
  will be part of $\measuredangle[(0,0); \: y_m, y_{m-1}]$, which jumps to $\measuredangle[(0,0); \: x_m, x_{m-1}]$ in the $\bff$-domain (since $x_m = \varphi(y_m)$ by construction), and the result follows by Corollary \ref{C:JumpLint}.

  Having checked \eqref{e:J1} at the two extreme (bottom and top) regions of $S_m(\bfg)$,  we can observe in Figures \ref{fig:SStack} and \ref{fig:Jstrip} that for $m\geq 3$ there is a third type of region which must be checked:

  At every point $G^m_k$ along $\breve{\Gamma}_m$ with
  $3\leq k \leq m-1,$
  we have two horizontal strips as in Figure \ref{fig:SStackg}: the two horizontal strips determined by $G_k^m$ (on $\breve{\Gamma}_m$) and its neighbors $G_{k-1}^{m-1}$ and $G_{k-2}^{m-1}$ (on $\breve{\Gamma}_{m-1}$).

  For both types of such regions, it will also follow from Corollary \ref{C:JumpLint} that \eqref{e:J1} is satisfied:
  In case of the top strip, let us call this trapezoid $R_{\text{top}}$ it is contained in the angle sector determined by the line segments $[G_k^m \to G_{k-1}^m]$ and $[G_{k-1}^{m-1} \to G_{k-2}^{m-1}]$:
  $$
  R_{\text{top}} \subset \measuredangle[p_0; \: \tfrac{1}{2^k-2}, \tfrac{1}{2^{k-1}-2}], \text{ where } p_0=(0, m-k+1).
  $$
  Its image $J(R_{\text{top}})$ in $S_m(\bff)$ is contained in the angle sector determined by the line segments $[F_k^m \rightarrow F_{k-1}^m]$ and $[F_{k-1}^{m-1} \rightarrow F_{k-2}^{m-1}]$:
  $$
  J(R_{\text{top}}) \subset \measuredangle[p_0; \: \tfrac{1}{2^k-1}, \tfrac{1}{2^{k-1}-1}]
  = \measuredangle[J(p_0); \: \varphi(\tfrac{1}{2^k-2}), \varphi(\tfrac{1}{2^{k-1}-2})] =
  J\bigg( \measuredangle[p_0; \: \tfrac{1}{2^k-2}, \tfrac{1}{2^{k-1}-2}] \bigg).
  $$

  The bottom strip, called $R_{\text{bot}}$ is contained in
  $\measuredangle[p_0; \: \tfrac{1}{2^{k+1}-2}, \tfrac{1}{2^{k-1}-2}]$,  where  $p_0 = \big(\tfrac{1}{3\cdot 2^{k-1}}, \tfrac{2^{k-1}-2}{3\cdot 2^{k-1}} +m-k+2\big)$,
  and its image $J(R_{\text{bot}})$ in $S_m(\bff)$ is contained in the angle sector
  $$
  \measuredangle[J(p_0); \: \tfrac{1}{2^{k+1}-1}, \tfrac{1}{2^{k-1}-1}] = J\bigg( \measuredangle[p_0; \: \tfrac{1}{2^{k+1}-2}, \tfrac{1}{2^{k-1}-2}] \bigg).
  $$

  Finally, for $G_m^m$, the bottom strip is contained in
  $
  \measuredangle[p_0; \: \tfrac{1}{3\cdot 2^{m}-2}, \tfrac{1}{2^{m-1}-2}]$,  where $p_0 = \big(\tfrac{1}{5\cdot 2^{m-2}}, \: \tfrac{3\cdot 2^m-2}{5\cdot 2^{m-2}}\big)$,
  and its image in $S_m(\bff)$ is contained in the angle sector
  $$
  \measuredangle[J(p_0); \: \tfrac{1}{3\cdot 2^m-1}, \tfrac{1}{2^{m-1}-1}] = J\bigg( \measuredangle[p_0; \: \tfrac{1}{3\cdot 2^{m}-2}, \tfrac{1}{2^{m-1}-2}] \bigg).
  $$
\end{proof}


\subsection{Concavity on the boundary}

\begin{theorem}
\label{T:bff-concave}
  For every fixed $\lambda>0$, the function $\candidatef(\cdot, \lambda)$ is concave.
\end{theorem}

In this subsection we will fix $\lambda > 0$ and write
$$
\candidatef(x) = \candidatef(x, \lambda)
$$
in order to make the notation lighter.

First note that, by construction, $\candidatef(x)$ is piecewise linear for every $\lambda >0$. Denote its vertices by
  $$
  0 \to \ldots \to p_m(\lambda) \to p_{m-1}(\lambda) \to \ldots \to p_1(\lambda) \to p_0(\lambda) = 1,
  $$
  and let $s_m(\lambda)$ denote the slope of $y = \candidatef(x)$ for $p_{m+1}(\lambda) \leq x \leq p_m(\lambda)$, for every $m\geq 0$.
  
  To show $\candidatef$ is concave, \textit{it suffices to show that the sequence of slopes $\{s_m(\lambda)\}_{m\geq 0}$ is non-decreasing}.
  We will look at this sequence for $\lambda \in (0,2]$, then $\lambda \in (2,3]$, and finally we generalize the latter case to $\lambda \in (k, k+1]$ for an integer $k\geq 3$. 
  
  Recall that we has defined the points $F_k^m$, $0\leq k\leq m$, living on the curve $\Gamma_m$, with coordinates:
  $$
  F_k^m = \big(\tfrac{1}{2^k}, \: \lambda_k^m\big), \text { where } \lambda_k^m := m-k+3-\tfrac{1}{2^k}.
  $$
  We will frequently use the equations of the lines determined by the line segments $[0 \to F_m^m]$: $x=x_m\lambda$, and by the line segments:
  $$
  [F_k^m \to F^m_{k-1}], \: 1\leq k \leq m: \:\: x = \tfrac{1}{2^k-1}[\lambda-(m-k+2)].
  $$
  Furthermore, we can see from \eqref{e:f-val-sm} that for $(x,\lambda)\in S_{m+1}(\bff)$ with $0<\lambda\leq m+2$, the slope of $y = \candidatef(x)$ is given by
  \begin{equation}
    \label{e:Cpr-slope}
    \frac{1}{2^{m+1}(\alpha_m(\lambda) - \alpha_{m+1}(\lambda))}.
  \end{equation}

\begin{proof}[Proof of Theorem \ref{T:bff-concave}.]
  Suppose now that $0<\lambda\leq 2$. The sequence of slopes of $y = \candidatef(x)$ is of the form
  $$
  0, s_1(\lambda), s_2(\lambda), \ldots, \text{ where } s_m(\lambda) = \frac{1}{2^m \lambda (x_{m-1}-x_m)} \text{ for every } m \geq 1,
  $$
  which is easily seen to be non-decreasing.

Now let $2<\lambda\leq 3$ (see Figure \ref{fig:Cpr23}). The breaking points are introduced by every $F_m^m$, $m\geq 1$, which all lie on the line $\lambda=3-x$. 
For every strip $S_{m+1}(\bff)$, $m\geq 1$, we have three possible expressions for the slope of $y = \candidatef(x)$:
\begin{align*}
    \text{For } 2<\lambda\leq \lambda_m^m:&\quad 2^{m+1} \lambda \big(\frac{1}{3\cdot 2^m -1} - \frac{1}{3\cdot 2^{m+1}-1}\big) =: \frac{1}{a_m(\lambda)}.\\
    \text{For } \lambda_m^m<\lambda\leq \lambda_{m+1}^{m+1}:&\quad 2^{m+1}  \big(\frac{\lambda-2}{ 2^m -1} - \frac{\lambda}{3\cdot 2^{m+1}-1}\big) =: \frac{1}{b_m(\lambda)}.\\
    \text{For } \lambda_{m+1}^{m+1}<\lambda\leq 3:&\quad 2^{m+1}  (\lambda-2) \big(\frac{1}{ 2^m -1} - \frac{1}{2^{m+1}-1}\big) =: \frac{1}{c_m(\lambda)}.
\end{align*}

\begin{figure}[h]
\begin{center}
\includegraphics[width=0.6\linewidth]{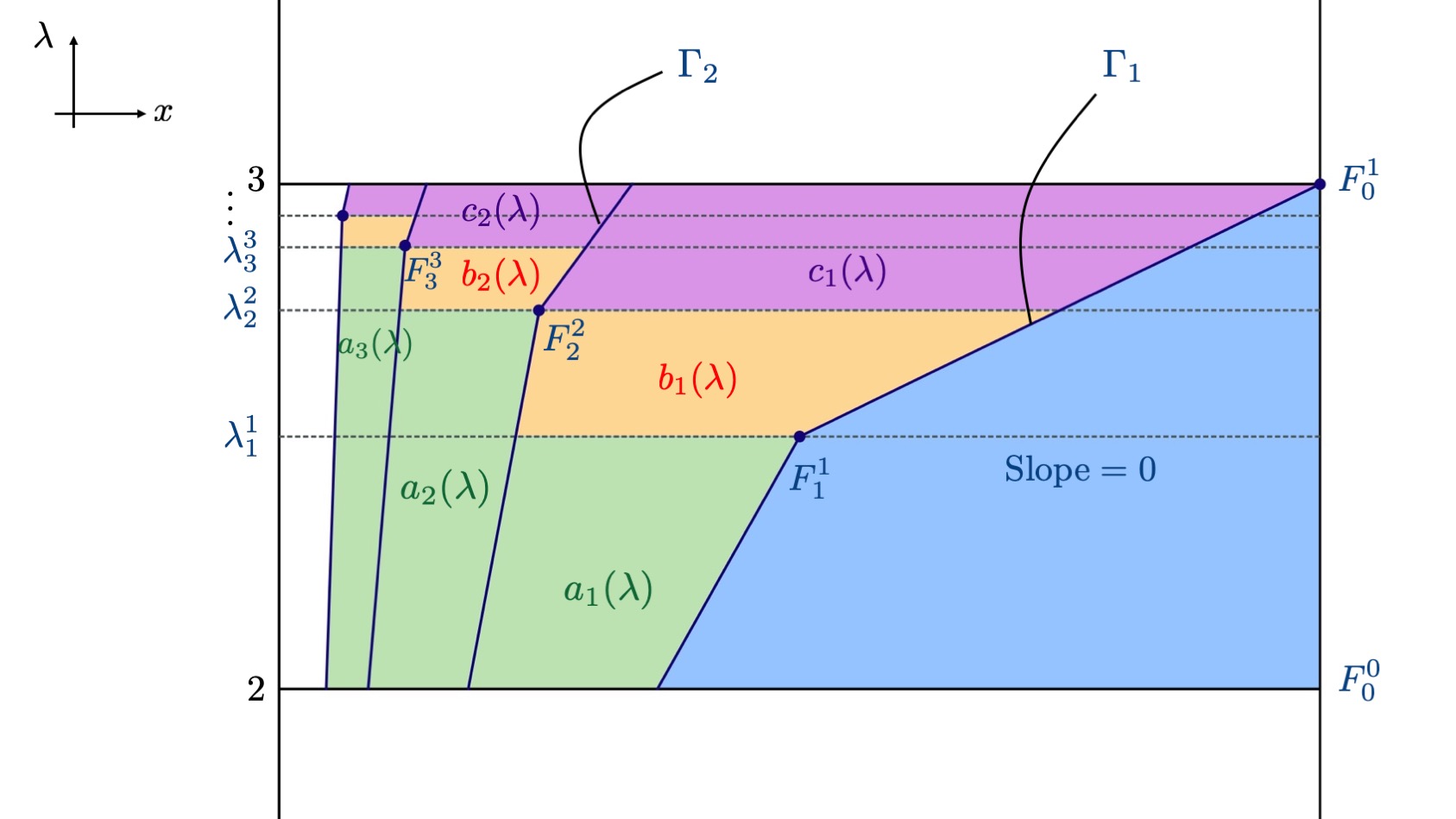}
\end{center}
       \caption{Slopes in the $2<\lambda\leq 3$ range.}
      \label{fig:Cpr23}
\end{figure}

Look at the sequence of slopes $\{s_m(\lambda)\}_{m\geq 0}$  for various ranges of $\lambda$:
	\begin{equation*}
	\begin{array}{rlllllll}
	2<\lambda\leq\lambda_1^1: & \textcolor{Azure}{0}, & \textcolor{indiagreen}{a_1(\lambda)},& \textcolor{indiagreen}{a_2(\lambda)},& \textcolor{indiagreen}{a_3(\lambda)},& \textcolor{indiagreen}{a_4(\lambda)},& 
	\textcolor{indiagreen}{a_5(\lambda)}, \ldots\\
	\lambda_1^1<\lambda\leq\lambda_2^2: & \textcolor{Azure}{0}, & \textcolor{orange-red}{b_1(\lambda)},& \textcolor{indiagreen}{a_2(\lambda)},& \textcolor{indiagreen}{a_3(\lambda)},& 
	\textcolor{indiagreen}{a_4(\lambda)},& \textcolor{indiagreen}{a_5(\lambda)}, \ldots\\
	\lambda_2^2<\lambda\leq\lambda_3^3: & \textcolor{Azure}{0}, & \textcolor{darkmagenta}{c_1(\lambda)},& \textcolor{orange-red}{b_2(\lambda)},& \textcolor{indiagreen}{a_3(\lambda)},& 
	\textcolor{indiagreen}{a_4(\lambda)},& \textcolor{indiagreen}{a_5(\lambda)}, \ldots\\
	\lambda_3^3<\lambda\leq\lambda_4^4: & \textcolor{Azure}{0}, & \textcolor{darkmagenta}{c_1(\lambda)},& \textcolor{darkmagenta}{c_2(\lambda)},& \textcolor{orange-red}{b_3(\lambda)},& 
	\textcolor{indiagreen}{a_4(\lambda)},& \textcolor{indiagreen}{a_5(\lambda)}, \ldots\\
	\lambda_4^4<\lambda\leq\lambda_5^5: & \textcolor{Azure}{0}, & \textcolor{darkmagenta}{c_1(\lambda)},& \textcolor{darkmagenta}{c_2(\lambda)},& \textcolor{darkmagenta}{c_3(\lambda)},& 
	\textcolor{orange-red}{b_4(\lambda)},& \textcolor{indiagreen}{a_5(\lambda)}, \ldots\\
	\vdots & &&\vdots&&&
	\end{array}
	\end{equation*}

It is easy to check that the sequences $\{a_m(\lambda)\}_{m\geq 1}$ and $\{c_m(\lambda)\}_{m\geq 1}$ are non-decreasing.
Therefore, in order to prove $\candidatef(x)$ is concave for $2<\lambda\leq 3$, we only need to show that:
	\begin{equation}
	\label{e:C23c1}
	b_m(\lambda) \leq a_{m+1}(\lambda), \:\text{ for } \lambda_m^m <\lambda \leq \lambda_{m+1}^{m+1}, \: m\geq 1,
	\end{equation}
and that:
	\begin{equation}
	\label{e:C23c2}
	c_m(\lambda) \leq b_{m+1}(\lambda), \:\text{ for } \lambda_{m+1}^{m+1} <\lambda \leq \lambda_{m+2}^{m+2}, \: m \geq 1.
	\end{equation}

To prove \eqref{e:C23c1}, we note that it is equivalent to
	$$
	\frac{\lambda-2}{\lambda} \geq \frac{(2^m-1)(3\cdot 2^{m+3}-1)}{(3\cdot 2^{m+1}-1)(3\cdot 2^{m+2}-1)}.
	$$
From $\lambda_m^m <\lambda \leq \lambda_{m+1}^{m+1}$, we have that
	$$
	\frac{\lambda-2}{\lambda} > \frac{2^m-1}{3\cdot 2^m -1} \geq \frac{(2^m-1)(3\cdot 2^{m+3}-1)}{(3\cdot 2^{m+1}-1)(3\cdot 2^{m+2}-1)}.
	$$
Similarly,  \eqref{e:C23c2} is equivalent to
	$$
	\frac{\lambda}{\lambda-2} \geq \frac{(2^{m-1}-1)(3\cdot 2^{m+2}-1)}{(2^{m+1}-1)(2^m-1)},
	$$
and the range condition $ \lambda_{m+1}^{m+1} <\lambda \leq \lambda_{m+2}^{m+2}$ on $\lambda$ gives that
	$$
	\frac{\lambda}{\lambda-2} \geq \frac{3\cdot 2^{m+2}-1}{2^{m+2}-1} \geq \frac{(2^{m-1}-1)(3\cdot 2^{m+2}-1)}{(2^{m+1}-1)(2^m-1)}.
	$$

  Now suppose $k<\lambda\leq k+1$ for some integer $k\geq 3$ (see Figure \ref{fig:Cpr9}). The breaking points are introduced by
  $\{F^m_{m-k+2}\}_{m\geq k-1}.$
  For every strip $S_{m+1}(\bff)$, $m\geq k-1$, we have three possible expressions for the slope of $y = \candidatef(x)$:
  \begin{align*}
      \text{For } k<\lambda\leq \lambda_{m-k+2}^m: &\quad 2^{m+1} \big(\lambda - (k-1) \big) \big(\frac{1}{2^{m-k+3} -1} - \frac{1}{2^{m-k+4}-1}\big) =: \frac{1}{a_m(\lambda)}.\\
      \text{For } \lambda_{m-k+2}^m<\lambda\leq \lambda_{m-k+3}^{m+1}: &\quad
      2^{m+1}  \big(\frac{\lambda-k}{ 2^{m-k+2} -1} - \frac{\lambda-(k-1)}{2^{m-k+4}-1}\big) =: \frac{1}{b_m(\lambda)}.\\
      \text{For } \lambda_{m-k+3}^{m+1}<\lambda\leq k+1: &\quad
      2^{m+1}  (\lambda-k) \big(\frac{1}{ 2^{m-k+2} -1} - \frac{1}{2^{m-k+3}-1}\big) =: \frac{1}{c_m(\lambda)}.
  \end{align*}

\begin{figure}[H]
\begin{center}
\includegraphics[width=0.6\linewidth]{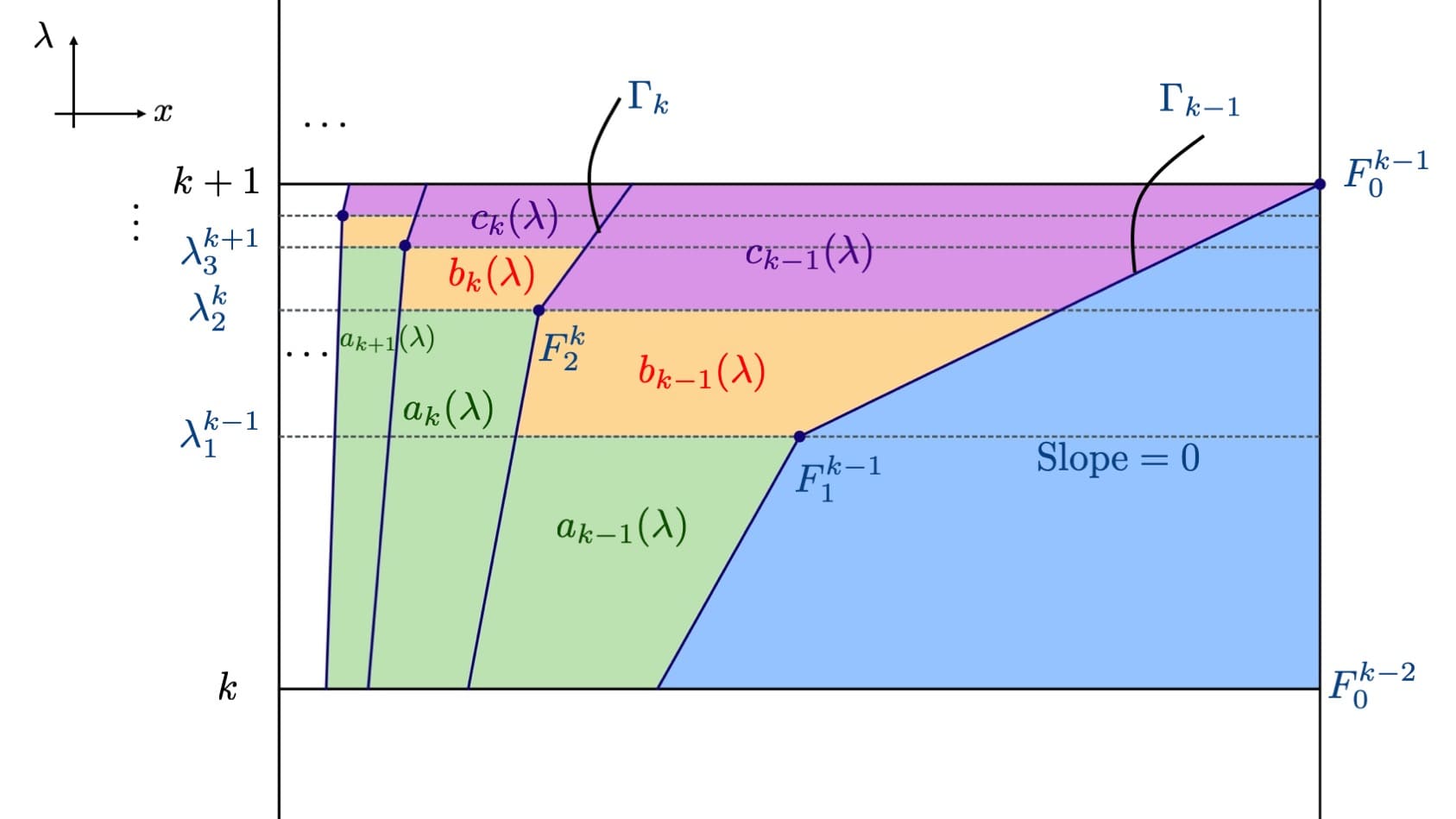}
\end{center}
       \caption{Slopes in the $k<\lambda\leq k+1$ range.}
      \label{fig:Cpr9}
\end{figure}

  The sequence of slopes of $y=\candidatef(x)$, $k<\lambda\leq k+1$, is of the form:
  \begin{equation*}
    \begin{array}{rllll}
      k<\lambda\leq\lambda_{k-1}^1: & \textcolor{Azure}{0}, & \textcolor{indiagreen}{a_{k-1}(\lambda)},& \textcolor{indiagreen}{a_k(\lambda)},& \textcolor{indiagreen}{a_{k+1}(\lambda)}, \ldots\\
      \lambda_1^{k-1}<\lambda\leq\lambda_2^k: & \textcolor{Azure}{0}, & \textcolor{orange-red}{b_{k-1}(\lambda)},& \textcolor{indiagreen}{a_k(\lambda)},& \textcolor{indiagreen}{a_{k+1}(\lambda)}, \ldots\\
      \lambda_2^k<\lambda\leq\lambda_3^{k+1}: & \textcolor{Azure}{0}, & \textcolor{darkmagenta}{c_{k-1}(\lambda)},& \textcolor{orange-red}{b_k(\lambda)},& \textcolor{indiagreen}{a_{k+1}(\lambda)}, \ldots\\
      \lambda_3^{k+1}<\lambda\leq\lambda_4^{k+2}: & \textcolor{Azure}{0}, & \textcolor{darkmagenta}{c_{k-1}(\lambda)},& \textcolor{darkmagenta}{c_k(\lambda)},& \textcolor{orange-red}{b_{k+1}(\lambda)}, \ldots\\
      \vdots & &&\vdots&
    \end{array}
  \end{equation*}
  It is easy to show that the sequences $\{a_m(\lambda)\}_{m\geq k-1}$, $\{c_m(\lambda)\}_{m\geq k-1}$ are non-decreasing. We are left to prove only that
  \begin{equation}
    \label{e:Ckc1}
    b_m(\lambda) \leq a_{m+1}(\lambda), \:\text{ for } \lambda_{m-k+2}^m <\lambda \leq \lambda_{m-k+3}^{m+1}, \: m\geq k-1,
  \end{equation}
  and that:
  \begin{equation}
    \label{e:Ckc2}
    c_m(\lambda) \leq b_{m+1}(\lambda), \:\text{ for } \lambda_{m-k+3}^{m+1} <\lambda \leq \lambda_{m-k+4}^{m+2}, \: m \geq k-1.
  \end{equation}
  The first condition \eqref{e:Ckc1} is equivalent to
  $$
  \frac{\lambda-k}{\lambda-(k-1)} \geq \frac{(2^{m-k+2}-1)(2^{m-k+6}-1)}{(2^{m-k+4}-1)(2^{m-k+5}-1)},
  $$
  and the range condition on $\lambda$ gives
  $$
  \frac{\lambda-k}{\lambda-(k-1)} > \frac{2^{m-k+2}-1}{2^{m-k+3}-1}\geq \frac{(2^{m-k+2}-1)(2^{m-k+6}-1)}{(2^{m-k+4}-1)(2^{m-k+5}-1)}.
  $$
  Similarly, \eqref{e:Ckc2} is equivalent to
  $$
  \frac{\lambda-(k-1)}{\lambda-k} \geq \frac{(2^{m-k+1}-1)(2^{m-k+5}-1)}{(2^{m-k+3}-1)(2^{m-k+2}-1)},
  $$
  and from the range condition on $\lambda$ we have
  $$
  \frac{\lambda-(k-1)}{\lambda-k} \geq \frac{2^{m-k+5}-1}{2^{m-k+4}-1}\geq \frac{(2^{m-k+1}-1)(2^{m-k+5}-1)}{(2^{m-k+3}-1)(2^{m-k+2}-1)},
  $$
  which concludes the proof.
\end{proof}

\subsection{Reducing the proof of $\bfb\in\mcb$ to $\bff$.} 
\label{sS:reduction}

In this subsection we show that $\candidateB \in \mathcal{B}$.

 \vspace{0.1in}
\noindent \textbf{Concavity:} Let us first prove that $\bfb(\cdot, \cdot, \lambda)$ is concave. 
We detail the argument for $\lambda>1$ (the case $0<\lambda\leq 1$ follows by similar considerations). 
So fix $\lambda>1$ and write $\bfb_{\lambda}(x,A)$ for $\bfb(x, A, \lambda)$:
 \begin{align*}
  \bfb_{\lambda}(x,A) = \begin{cases}
    \frac{A}{2}\bff\big(\frac{2x}{A}, \lambda\big), &\text{ if }  2x \leq A\\
    \frac{A}{2}\bff(1,\lambda), & \text{ otherwise.}
  \end{cases}
\end{align*}
Observe two crucial facts: first, $\bfb_{\lambda}$ is homogeneous:
	\begin{equation}
	\label{e:Chom}
	\bfb_{\lambda}(tx, tA) = t \cdot \bfb_{\lambda}(x, A).
	\end{equation}
Second,
 the function $\bfb_{\lambda}(\cdot, A)$ is \textit{concave}.

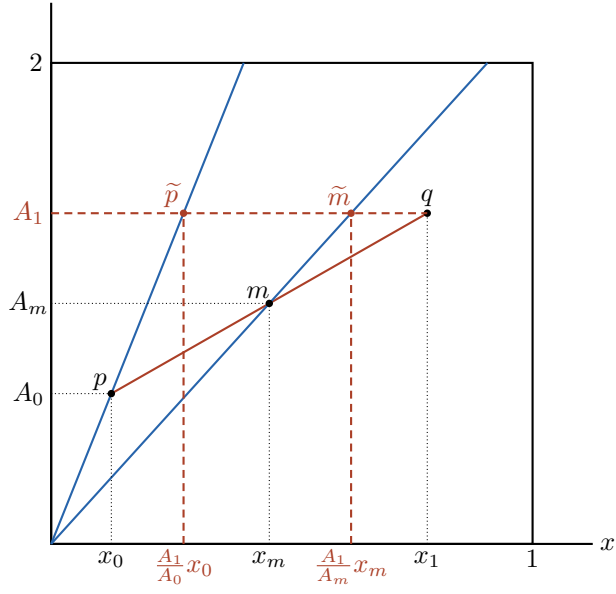
\begin{figure}[H]
  \begin{tikzpicture}[scale=0.2]
    \draw[thick, -] (0, 0) -- (36, 0);
    \draw[thick, -] (0, 0) -- (0, 36);
    \draw[thick, -] (0, 0) -- (32, 0) -- (32, 32) -- (0,32) -- cycle;

    \coordinate (0) at (0,0);

    \coordinate (P) at (4, 10);
    \coordinate (Q) at (25, 22);
    \coordinate (M) at ($0.5*(P) + 0.5*(Q)$);
    \draw[thick, -, irinared] (P) -- (Q);

    \node at ($(0, 10) + (-1.5, 0)$) {$A_0$};
    \node at ($(0, 16) + (-1.5, 0)$) {$A_m$};
    \node[irinared] at ($(0, 22) + (-1.5, 0)$) {$A_1$};

    \node at ($(4, 0) + (0, -1)$) {$x_0$};
    \node at ($(25, 0) + (0, -1)$) {$x_1$};
    \node at ($0.5*(29, 0) + (0, -1)$) {$x_m$};

    \coordinate (Pe) at ($32/10*(P)$);
    \coordinate (Me) at ($32/16*(M)$);

    \draw[thick, -, irinablue] (0) -- (Pe);
    \draw[thick, -, irinablue] (0) -- (Me);

    \coordinate (Pt) at ($22/10*(P)$);
    \coordinate (Mt) at ($22/16*(M)$);

    \node[irinared] at ($22/10*(4, 0) + (0, -1.5)$) {$\frac{A_1}{A_0}x_0$};
    \node[irinared] at ($11/16*(29, 0) + (0, -1.5)$) {$\frac{A_1}{A_m}x_m$};

    \node at (37, 0) {$x$};
    \node at (32, -1) {$1$};
    \node at (-1, 32) {$2$};

    \node at ($(P) + (-0.75, 0.75)$) {$p$};
    \fill (P) circle (7pt);

    \node at ($(Q) + (-0.0, 1)$) {$q$};
    \fill (Q) circle (7pt);

    \node at ($(M) + (-0.75, 0.75)$) {$m$};
    \fill (M) circle (7pt);

    \node[irinared] at ($(Pt) + (-0.8, 1.2)$) {$\widetilde{p}$};
    \fill[irinared] (Pt) circle (7pt);

    \node[irinared] at ($(Mt) + (-0.8, 1.2)$) {$\widetilde{m}$};
    \fill[irinared] (Mt) circle (7pt);

    \draw[thick, densely dashed, irinared] let \p1 = (Q) in (0, \y1) -- (Q);
    \draw[thin, densely dotted] let \p1 = (Q) in (\x1, 0) -- (Q);

    \draw[thin, densely dotted] let \p1 = (M) in (0, \y1) -- (M);
    \draw[thin, densely dotted] let \p1 = (M) in (\x1, 0) -- (M);

    \draw[thin, densely dotted] let \p1 = (P) in (0, \y1) -- (P);
    \draw[thin, densely dotted] let \p1 = (P) in (\x1, 0) -- (P);

    \draw[thick, densely dashed, irinared] let \p1 = (Pt) in (\x1, 0) -- (Pt);
    \draw[thick, densely dashed, irinared] let \p1 = (Mt) in (\x1, 0) -- (Mt);

  \end{tikzpicture}
       \caption{Elements of the proof that $\candidateB_\lambda$ is concave.}
      \label{fig:lastone}
\end{figure}

Since $\bfb_{\lambda}$ is bounded, it suffices to show it is midpoint concave. So let two points in the domain (see Figure \ref{fig:lastone}):
	$$
	p = (x_0, A_0); \:\: q= (x_1, A_1),
	$$
and let $m=(x_m, A_m)$ be their midpoint. We show that
	\begin{equation}
	\label{e:conc}
	\bfb_{\lambda}(m) \geq \tfrac{1}{2}\big(\bfb_{\lambda}(p)+\bfb_{\lambda}(q)\big).
	\end{equation}
As in Figure \ref{fig:lastone}, consider the points
	$$
	\widetilde{p}:= \big(\tfrac{A_1}{A_0}x_0, A_1\big); \:\: \widetilde{m}:=\big(\tfrac{A_1}{A_m}x_m, A_1\big).
	$$
By \eqref{e:Chom}, we can express:
	$$
	\bfb_{\lambda}(p) = \tfrac{A_0}{A_1}\bfb_{\lambda}(\widetilde{p}); \:\: 	\bfb_{\lambda}(m) = \tfrac{A_m}{A_1}\bfb_{\lambda}(\widetilde{m}).
	$$
Then \eqref{e:conc} can be written as:
	$$
	\bfb_{\lambda}\big(\tfrac{A_1}{A_m}x_m, A_1\big) \geq \tfrac{A_0}{A_0+A_1} \bfb_{\lambda}\big(\tfrac{A_1}{A_0}x_0, A_1\big) +
		\tfrac{A_1}{A_0+A_1} \bfb_{\lambda}(x_1, A_1),
	$$
which follows by the \textit{horizontal concavity} (in the $x$-direction) of $\bfb_{\lambda}$.

\vspace{0.2in}
\noindent \textbf{Jump Inequality:} Finally, we prove that $\bfb$ satisfies the Jump Inequality:
	\begin{equation}
	\label{e:Jumpagain}
	\bfb(x,A,\lambda) \leq \bfb(x,A+1,\lambda+x), \text{ for all } 0\leq A\leq 1.
	\end{equation}
We already know this holds at $A=0$ and $A=1$ (the latter is exactly \eqref{e:Ja2}), so we assume in what follows that $0<A<1$. We will frequently use the fact that
	\begin{equation}
	\label{e:flower}
	\bff\bigg(\frac{2x}{A+1}, \lambda+x\bigg) \geq \bff\bigg(\frac{2x}{A+1}, \lambda + x \frac{2}{A+1}\bigg)
	\geq \bfg\bigg(\frac{2x}{A+1}, \lambda\bigg),
	\end{equation}
where the first inequality follows because $\bff$ is non-increasing in $\lambda$, and the second from \eqref{e:Ja2}. 
The most frequent situation from here is when we can express $\bfg$ as $\bfg(x,\lambda) = \tfrac{1}{2}\bff(2x,\lambda)$ -- this is treated in \textit{Case 1A.1} below. However, $\bfg$ does not have this expression everywhere in the domain, so we must split into cases, depending on the expression of $\bfg$.
Another useful fact will be:
	\begin{equation}
	\label{e:star}
	\bff(tx,\lambda) \geq t \bff(x,\lambda)
	\end{equation}
which follows from concavity of $\bff(\cdot, \lambda)$.

\vspace{0.1in}
 \underline{\textit{Case 1:}} $0<\lambda\leq 1$. Since $\lambda+x$ could land in $(1,2]$, we have two major subcases. We will only show the first case below, as the other possibility ($1<\lambda+x\leq 2$) is proved by identical considerations.
 
 \vspace{0.1in}
 \noindent \underline{\textit{Case 1A:}} $0< x\leq 1-\lambda$. Then $0<\lambda+x\leq 1$, and the quantities in \eqref{e:Jumpagain} are explicitly given by:
 \begin{align*}
  \bfb(x,A,\lambda) = \begin{cases}
    \tfrac{A}{2}\bff\big(\tfrac{2x}{A}, \lambda\big), &\text{ if }  0<x\leq \tfrac{A}{4}\lambda\\
    \tfrac{A}{3} + \tfrac{2x}{3\lambda}, &\text{ if } \tfrac{A}{4}\lambda<x\leq A\lambda\\
    A, & \text{ if } A\lambda<x\leq 1,
  \end{cases}
\end{align*}
 and
 \begin{align*}
  \bfb(x,A+1,\lambda+x) = \begin{cases}
    \tfrac{A+1}{2}\bff\big(\tfrac{2x}{A+1}, \lambda+x\big), &\text{ if }  0<x\leq \tfrac{A+1}{3-A}\lambda\\
    \tfrac{A+1}{3} + \tfrac{2x}{3(\lambda+x)}, &\text{ if } \tfrac{A+1}{3-A}\lambda<x\leq \tfrac{2-A}{A}\lambda\\
    1, & \text{ if } \tfrac{2-A}{A}\lambda<x\leq 1.
  \end{cases}
\end{align*}

 \vspace{0.1in}
 \noindent \underline{\textit{Case 1A.1:}} $0<x\leq \tfrac{A}{4}\lambda$.
Then \eqref{e:Jumpagain} is equivalent to:
	\begin{equation}
	\label{e:faq}
	\tfrac{A}{A+1}\bff\big(\tfrac{2x}{A}, \lambda\big) \leq \bff\big(\tfrac{2x}{A+1}, \lambda+x\big).
	\end{equation}
We check this below:
$$
\bff\big(\tfrac{2x}{A+1}, \lambda+x\big) 
\underset{\eqref{e:flower}}\geq \bfg\big(\tfrac{2x}{A+1}, \lambda\big) = \tfrac{1}{2}\bff\big(\tfrac{4x}{A+1}, \lambda\big) 
=\tfrac{1}{2} \bff\big(\tfrac{2x}{A} \cdot \tfrac{2A}{A+1}, \lambda\big) 
\underset{\eqref{e:star}}\geq \tfrac{A}{A+1}\bff\big(\tfrac{2x}{A}, \lambda\big).
$$

 \vspace{0.1in}
 \noindent \underline{\textit{Case 1A.2:}} $\tfrac{A}{4}\lambda <x\leq \tfrac{A+1}{8}\lambda$. 
 Then \eqref{e:Jumpagain} is equivalent to:
	\begin{equation}
	\label{e:eh1}
	\tfrac{A}{3} + \tfrac{2x}{3\lambda} \leq \tfrac{A+1}{2} \bff\big(\tfrac{2x}{A+1}, \lambda+x\big).
	\end{equation}
Since $0<\lambda\leq1$:
$$
\tfrac{A+1}{2} \bff\big(\tfrac{2x}{A+1}, \lambda+x\big) 
\underset{\eqref{e:flower}}\geq \tfrac{A+1}{2} \bfg\big(\tfrac{2x}{A+1}, \lambda\big) 
= \tfrac{A+1}{4}\bff\big(\tfrac{4x}{A+1}, \lambda\big)
= \tfrac{A+1}{4}\bff\big(\tfrac{8x}{(A+1)\lambda}\cdot \tfrac{\lambda}{2}, \lambda\big)
\underset{\eqref{e:star}}\geq \tfrac{2x}{\lambda} \underbrace{\bff\big(\tfrac{\lambda}{2}, \lambda\big)}_{=1}.
$$
Now, from the range assumption in this case, we can further say that
$$
\tfrac{A+1}{2} \bff\big(\tfrac{2x}{A+1}, \lambda+x\big) \geq \tfrac{2x}{\lambda} = \tfrac{2x}{3\lambda} + \tfrac{4}{3\lambda}x
> \tfrac{2x}{3\lambda} + \tfrac{4}{3\lambda} \tfrac{A\lambda}{4} = \tfrac{A}{3} + \tfrac{2x}{3\lambda},
$$
proving the claim.

 \vspace{0.1in}
 \noindent \underline{\textit{Case 1A.3:}} $\tfrac{A+1}{8}\lambda <x\leq A\lambda$. In this case, \eqref{e:Jumpagain} is again equivalent to \eqref{e:eh1}, the only difference in the argument being the expression of $\bfg$:
 $$
\tfrac{A+1}{2} \bff\big(\tfrac{2x}{A+1}, \lambda+x\big) 
\underset{\eqref{e:flower}}\geq \tfrac{A+1}{2} \bfg\big(\tfrac{2x}{A+1}, \lambda\big) 
= \tfrac{A+1}{2} \big(\tfrac{1}{3}+\tfrac{4x}{3\lambda(A+1)} \big)= \tfrac{A+1}{6} + \tfrac{2x}{3\lambda} > \tfrac{A}{3} + \tfrac{2x}{3\lambda}.
$$
 The rest of the cases as $A\lambda<x\leq 1$  follow by trivial calculations.

\vspace{0.1in}
 \noindent \underline{\textit{Case 2:}} $\lambda > 1$

Define the extension of $\candidatef$
\begin{align*}
  \widetilde{\candidatef}(x,\lambda) = \begin{cases}
    \candidatef(x,\lambda) &\text{if } x \leq 1 \\
    \candidatef(1,\lambda) &\text{if } x \geq 1.
  \end{cases}
\end{align*}
It is easy to check that $\widetilde{\candidatef}(\cdot, \lambda)$ is concave and non-increasing in $\lambda$.

When $\lambda > 1$ we have
\begin{align*}
  \candidateB(x,A,\lambda) = \frac{A}{2}\widetilde{\candidatef}\Big(\frac{2x}{A}, \lambda\Big).
\end{align*}
Since $\candidateB$ satisfies jump when $A=1$ we have, for all $0 \leq x \leq 1$
\begin{align*}
  \widetilde{\candidatef}(x,\lambda+x) &= \candidatef(x,\lambda+x) \\
  &\geq \candidateB(x, 1, \lambda) \\
  &= \frac{1}{2}\widetilde{\candidatef}(2x, \lambda).
\end{align*}
That is:
\begin{align} \label{guillej}
  \widetilde{\candidatef}(x,\lambda+x) \geq \frac{1}{2}\widetilde{\candidatef}(2x, \lambda).
\end{align}

We just need to show that, for $0 < A < 1$
\begin{align} \label{guillegoal}
  \frac{A+1}{2}\widetilde{\candidatef}\Big(\frac{2x}{A+1}, \lambda + x\Big) \geq \frac{A}{2}\widetilde{\candidatef}\Big(\frac{2x}{A},\lambda\Big).
\end{align}

Since $\widetilde{\candidatef}$ is non-increasing in $\lambda$
\begin{align*}
  \widetilde{\candidatef}\Big(\frac{2x}{A+1}, \lambda + x\Big) \geq \widetilde{\candidatef}\Big(\frac{2x}{A+1}, \lambda + \frac{2x}{A+1}\Big).
\end{align*}

Now we can proceed as follows:
\begin{align*}
  \widetilde{\candidatef}\Big(\frac{2x}{A+1}, \lambda + x\Big) &\geq \widetilde{\candidatef}\Big(\frac{2x}{A+1}, \lambda + \frac{2x}{A+1}\Big) \\
  &\geq \frac{1}{2}\widetilde{\candidatef}\Big(\frac{4x}{A+1}, \lambda\Big) \tag{by \eqref{guillej}} \\
  &\geq \frac{A}{A+1}\widetilde{\candidatef}\Big(\frac{2x}{A}, \lambda\Big) \tag{by concavity of $\widetilde{\candidatef}$}
\end{align*}
and this is equivalent to \eqref{guillegoal}.

\bibliography{bibliography}
\bibliographystyle{abbrv}

\end{document}